\documentclass[12pt]{amsart}
\usepackage{amsmath,amssymb,amsthm}
\usepackage{color}
\usepackage[colorlinks, linkcolor=red, citecolor=blue, urlcolor=blue, hypertexnames=false]{hyperref}

\newtheorem{lemma}{Lemma}[section]
\newtheorem{proposition}{Proposition}[section]
\newtheorem{corollary}{Corollary}[section]

\newtheorem{theorem}{Theorem}[section]

\theoremstyle{definition}
\newtheorem{definition}{Definition}[section]

\theoremstyle{remark}

\theoremstyle{remark}
\newtheorem{remark}{Remark}[section]
\numberwithin{equation}{section}

\newcommand{\N}{{\mathbb N}}

\newcommand{\R}{{\mathbb R}}

\definecolor{blu}{rgb}{0,0,1}

\title[Ground state solutions of $(2,q)$-Laplacian Schr\"odinger equations]{Ground state solutions of a class of $(2,q)$-Laplacian Schr\"odinger equations with inhomogeneous nonlinearity}

\author[Ying Huang]{Ying Huang}
\address{Ying Huang
\newline\indent
School of Mathematics and Information Sciences, 
\newline\indent
Guangzhou University, Guangzhou, China}

\author[Tingjian Luo]{Tingjian Luo}
\address{Tingjian Luo
\newline\indent
School of Mathematics and Information Sciences, 
\newline\indent
Guangzhou University, Guangzhou, China}
\email{luotj@gzhu.edu.cn}

\author[Youde Wang]{Youde Wang}
\address{Youde Wang
\newline\indent
School of Mathematics and Information Sciences, 
\newline\indent
Guangzhou University, Guangzhou, China;
\newline\indent
Institute of Mathematics, Academy of Mathematics and  Systems Science, Chinese Academy of Sciences, Beijing,China;
\newline\indent
School of Mathematical Sciences, University of Chinese Academy of Sciences, Beijing, China}
\email{wyd@math.ac.cn}

\begin{document}
\subjclass[2000]{35J50, 35Q41, 35Q55, 37K45}

\keywords{Ground state solutions, $(2,q)$-Laplacian Schr\"odinger equations, Existence and multiplicity.}

\begin{abstract} In this paper, we systematically investigate the ground state solutions of a class of $(2,q)$-Laplacian Schr\"odinger equations with inhomogeneous nonlinearity. By analyzing global and local constrained variational problems, we establish the existence, non-existence, and asymptotic behavior of ground states, addressing the mass-subcritical, mass-critical, and mass-supercritical regimes. As a byproduct, we prove a multiplicity of bound states with prescribed mass. Some of our existence results are sharp. The proofs are based primarily on constrained variational techniques.
\end{abstract}
\maketitle

\section{Introduction}
In the paper,we are concerned with the following $(2,q)$-Laplacian Schr\"odinger equations with inhomogeneous nonlinearity:
 \begin{equation}\label{eq1.1}
 -\triangle u -\triangle_q u =\lambda u + \frac{|u|^{p-2}u}{|x|^b}, \quad  x\in \mathbb{R}^N,
\end{equation}
where $\triangle_q u = div (|\nabla u|^{q-2} \nabla u )$, with $q>2$, $N\geq 1, 0<b<\min\{2,N\}$ and $2<p<\frac{q(N-b)}{(N-q)^+}$. Here, $\frac{q(N-b)}{(N-q)^+}:=\frac{q(N-b)}{N-q}$ if $q<N$, and $\frac{q(N-b)}{(N-q)^+}:=+\infty$ if $q\geq N$. Such stationary equations stems from the study of the general reaction-diffusion system:
\begin{equation}\label{eq1.2}
\partial_t\phi (t,x)-\Delta_r \phi(t,x)-\Delta_q \phi(t,x)=f(x, \phi(t,x)), \quad (t,x)\in \R^+\times \R^N,
\end{equation}
which has been widely applied on the chemical reaction, plasma physics and biophysics. In those models, the wave function $\phi(t,x)$ describes a concentration; the $(r,q)$-Laplacian operator, defined as $-\Delta_r \phi - \Delta_q \phi$, arises in the physical systems where nonstandard diffusion processes occur. Such equations model phenomena with double-phase heterogeneity, where two distinct diffusion mechanisms (with exponents $r$ and $q$) compete or coexist \cite{A1,CR}. For instance, in quantum physics, they describe Bose-Einstein condensates with spatially dependent interactions, where $r$ and $q$ reflect different regimes of particle dispersion \cite{D}. In nonlinear optics, $(r,q)$-Laplacian equations govern electromagnetic wave propagation in composite materials with anisotropic refractive indices. The inhomogeneous nonlinearity $f(x,\phi)$ further accounts for spatially localized potentials or nonuniform media \cite{A,F}. From a variational viewpoint, the interplay between $r$- and $q$-growth conditions leads to rich mathematical structures, including bifurcations of ground states and multiplicity of solutions—key for understanding phase transitions in such systems. The $(r,q)$-framework thus bridges mathematical analysis with multiscale physical modeling, where nonhomogeneous diffusion is essential.

Considering the stationary solutions of Eq. \eqref{eq1.2}, leads to the following time-independent equation:
\begin{equation}\label{eq1.3}
-\Delta_r u(x)-\Delta_q u(x)=f(x, u), \quad x\in \R^N.
\end{equation}
In \cite{CI}, Cherfils and Il’yasov considered the existence and non-existence of solutions of Eq. \eqref{eq1.3} on bounded domain, which exposit some essential differences between the $(r,q)$-laplacians equation and single Laplacian equation. Since then, using mainly the variational method, Eq. \eqref{eq1.3} has been intensively studied. See for example, \cite{BBF1} for the existence of infinitely many weak solutions with negative energy, and \cite{BBF2} for infinitely many solutions with positive energy. See also the early work on the existence and multiplicity of week solutions under general assumption on the nonlinearity in \cite{HL1,HL2}, and also \cite{AR, GBZA}. The reference concerned with the study of $(p,q)$-Laplacian equations is huge and rich, here we refer the readers to the recent work \cite{BBF1,PRR3,ZZR} and the reference therein.

Let $r=2$, then Eq. \eqref{eq1.3} is reduced to a $(2,q)$-Laplacian equation. In bounded domain, with a Dirichlet boundary condition, Benci et. al. \cite{BMV} initially considered its eigenvalue problem. Later,  Papageorgiou et. al. established a serious of work on the existence and multiplicity of solutions with precise sign information, see \cite{PRR1,PRR2,PRR3}.

In this paper, we aim to investigate the solutions of Eq. \eqref{eq1.1} with a prescribed $L^2$-norm. Specifically, for a given mass $c>0$, we find a pair $(\lambda_c, u_c)$ that serves as a solution to Eq. \eqref{eq1.1}. Such solutions are of particular significance in certain physical applications. This is because the quantity represents the mass, which remains conserved throughout the time evolution described by Eq. \eqref{eq1.2}. In the existing literature, solutions with a prescribed $L^2$-norm are often referred as ``\emph{normalized solutions}". Over the past four decades, normalized solutions have garnered an increasingly growing amount of attention. By the variational method, when aiming to find a normalized solution of Eq. \eqref{eq1.1} for a given $c>0$, we consider the constrained critical points of the following functional:
\begin{equation}\label{func1}
 I(u)=\frac{1}{2}\|\nabla u\|_2^2 + \frac{1}{q}\|\nabla u\|_q^q - \frac{1}{p} \int \frac{|u|^p}{|x|^b}dx,
\end{equation}
on the $L^2$-constraint:
\begin{equation*}
 S(c):=\{u \in X: \int |u|^2 dx = c\},
\end{equation*}
where $X:=H^1(\R^N) \cap D^{1,q}(\mathbb{R}^N)$. In this case, the parameter $\lambda\in \R$ appears as the corresponding Lagrange Multiplier, which is a part of unknown. Here we fucus on the ground state solutions with prescribed $L^2$-norm, which is in the following sense:
\begin{definition}\label{def1.1}
We say that $u_c\in S(c)$ is \textbf{a ground state solution} to Eq. \eqref{eq1.1}, if it is a solution having minimal energy
among all the solutions which belong to $S(c)$. Namely,
$$I(u_c)=\inf\{I(u)\ | \ u\in S(c), (I|_{S(c)})'(u)=0\}.$$
\end{definition}

When $b=0$, the nonlinearity is homogeneous. The reference \cite{BY} seems to be the first trying on the study of $L^2$-normalized solutions for the $(2,q)$-Laplacian equation with homogeneous nonlinearity. Precisely, the authors use the concentration-compactness principle and the constrained variational method to prove the existence, non-existence, and multiplicity of normalized solutions of Eq. \eqref{eq1.1} with $b=0$. We also note a recent work \cite{DJP} for the study of $L^2$-normalized solutions for the $(2,q)$-Laplacian equation with logarithmic nonlinearity $f(x, u)=u \log |u|^2$. In \cite{DJP} , the authors established the existence and multiplicity of $L^2$-normalized solutions. 

When $b\neq 0$, the corresponding nonlinearity is referred as inhomogeneous one, which having more meanings in application, for example, in nonlinear optics, the inhomogeneous nonlinearity appears in some models as the corrections to the nonlinear power-law response \cite{A}. Also, in some case, the inhomogeneous nonlinearity causes new phenomena compared with the homogeneous case, see \cite{SBAC}, the inhomogeneous nonlinearity may induce to the spontaneous symmetry breaking of solutions. In view of the importance of the inhomogeneous case, our main motivation of this paper is to establish systematically the existence and multiplicity of the $L^2$-normalized solutions of Eq. \eqref{eq1.1}, under different assumptions with respect to the parameters in the equations. In analysis, as we shall see, it is not a analogous with the case $b=0$. In fact, we develop some new elements which may be used to improve some results in \cite{BY}. Also some of our existence results are described sharply.

Before giving our main results, let us recall the study of the following $L^q$-normalized problem (where $q>2$):
\begin{equation}\label{eq1.1-q}
\left\{
\begin{array}{l}
-\triangle u -\triangle_q u =\lambda |u|^{q-2}u + \frac{|u|^{p-2}u}{|x|^b}, \quad  x\in \mathbb{R}^N,\\
\int |u|^qdx=c,\ c>0.\\
\end{array}
\right.
\end{equation}
The first study of such type of problems should refer to \cite{GZZ}, where the authors proved the existence and asymptotic behavior of $L^q$-normalized solutions for a class of purely $p$-Laplacian equations with the homogeneous nonlinearity of mass-critical growth. Also recently, Cai and R\u{a}dulescu proved in \cite{CR} the existence of $L^q$-normalized solutions under more general assumption on the homogeneous nonlinearity. More studies on the $L^q$-normalized solutions for the equations involving the $p$-Laplcian can be referred to the reference given in \cite{CR}, see also a very recent work \cite{ZZL} for the case with power nonlinearity. We remark that, $\lambda |u|^{q-2}u$ as $q=2$ in Eq. \eqref{eq1.1-q} (namely the case in Eq. \eqref{eq1.1}) is called in the reference the chemical potential, which has more important applications, see e.g. \cite{C}. Also, the analysis on the compactness of the $L^2$-normalized problem is more complicated than the $L^q$-normalized one. In other words, using the same ideas of the proofs, one could extend the results in this paper without further essential difficulties for Eq. \eqref{eq1.1-q}.\\   

When it comes to apply the variational method, as an initial and fundamental step, we are required to establish the inhomogeneous $L^{q}$-Gagliardo-Nirenberg inequality with sharp constant, which is crucial to establish some sharp existence results in this paper.

\begin{theorem}\label{I-th-GN}
(Sharp inhomogeneous $L^{q}$-Gagliardo-Nirenberg inequality) Assume that $2<p<\frac{q(N-b)}{(N-q)^+}$, with $0 < b < \min\{2,N\}$, $N\geq1$, $q\geq 2 $, then  there exists a sharp constant $\mathcal{K}_{N,p,q}:=\frac{p}{\|Q_{p,q}\|_2^{p-2}}>0$, such that
\begin{equation}\label{I-p-GN}
  \int \frac{|u|^{p}}{|x|^{b}} dx \leq \mathcal{K}_{N,p,q} \|\nabla u\|_q^{\sigma_{p,q}}\|u\|_2^{p-\sigma_{p,q}}, \quad \forall u \in D^{1,q}(\R^N)\cap L^{2}(\mathbb{R}^{N}),
\end{equation}
where $\sigma_{p,q}:=\frac{q[N(p-2)+2b]}{N(q-2)+2q}$, and $Q_{p,q}$ is a fixed ground state solution of
\begin{equation}\label{I-eq2.1}
\sigma_{p,q} \triangle_q Q + (p-\sigma_{p,q}) Q - |x|^{-b}|Q|^{p-2}Q=0.
\end{equation}
Moreover, the equality in \eqref{I-p-GN} holds if $u=Q_{p,q}$. In particular, let $q=2$, then \eqref{I-p-GN} is reduced to 
\begin{equation}\label{I-p-GN2}
  \int \frac{|u|^{p}}{|x|^{b}} dx \leq \mathcal{K}_{N,p,2} \|\nabla u\|_2^{\frac{N(p-2)+2b}{2}}\|u\|_2^{\frac{2N+2b-(N-2)p}{2}}, \quad \forall u \in H^1(\R^N).
\end{equation}
\end{theorem}

Due to \eqref{I-p-GN}, one could verify easily that the functional $I(u)$ is of class $C^1$ on $S(c)$. Thus, for given $c>0$, to find a ground state type of critical point of the functional $I(u)$ on $S(c)$, when the functional is bounded from below, it is natural to consider the following global minimization problem:
\begin{eqnarray}\label{I-mini}
m(c):=\inf_{u\in S(c)}I(u), \ c>0.
\end{eqnarray}
It is standard to verify that a minimizer of $m(c)$ is indeed a ground state solution of Eq. \eqref{eq1.1}, see for example \cite{C}. As we shall see, the existence of minimizers of $m(c)$ strongly relies on the parameters $p,N,c$. 

Denote by $p_q^*:=\frac{2(q-b)}{N}+q$ the mass critical exponent which is in the sense that the functional is bounded from below as $p<p_q^*$ and unbounded from below as $p>p_q^*$, see Lemma \ref{lm3.1} and Remark \ref{rek3.1}, and let
\begin{equation}\label{I-c1}
c_1^*:=\inf\{c>0\mid m(c)<0\},
\end{equation}
then our first result is concerned with the sharp existence of minimizers of $m(c)$. 
\begin{theorem}[Mass subcritical case]\label{I-th-subcritical}
Assume that $2<p<p_q^*$, with $0 < b < \min\{2,N\}$, $N\geq1$, and $q>2$. Then,
\begin{itemize}
  \item [(1)] If $2<p<\frac{2(2-b)}{N}+2$, then for all $c>0$, $m(c)$ admits a minimizer;
  \item [(2)] If $p=\frac{2(2-b)}{N}+2$, then $0<c_1^*<+\infty$, and $m(c)$ admits a minimizer \textbf{if and only if} $c>c_1^*$;
  \item [(3)] If $\frac{2(2-b)}{N}+2< p <p_q^*$, then $0<c_1^*<+\infty$, and $m(c)$ admits a minimizer \textbf{if and only if} $c\geq c_1^*$.
\end{itemize}
In addition, any minimizer of $m(c)$ can be assumed to be non-negative, radially symmetric and radially decreasing with respect to some point.
\end{theorem}
\begin{remark}\label{rek1.1}
To prove the existence of minimizers of $m(c)$, it is essential to prove that $m(c)<0$, which helps to rule out the vanishing and dichotomy of a minimizing sequence of $m(c)$. Then the concentration-compactness principle derives the compactness of the minimizing sequence. Note particularly that in the critical case $c=c_1^*$, Lemma \ref{lm3.1} shows that $m(c_1^*)=0$, the former strategy does not work, instead, we use some approximation arguments, saying $c_k\searrow c_1^*$, with $m(c_k)<0$ having a minimizer $u_k$, then by showing the convergence of $u_k$, we have the limit function as a minimizer of $m(c_1^*)$. However, it should be distinguished the difference that, except for the critical case $c=c_1^*$, we have proved that any minimizing sequence of $m(c)$ is pre-compact, which then implies the orbital stability of solutions in the sense of Cazenave and Lions \cite{C}, for the corresponding Cauchy problem. 
\end{remark}
Concerned with the mass critical case, we prove the non-attain of $m(c)$. 
\begin{theorem}[Mass critical case]\label{I-th-critical}
Assume that $p=p_q^*$, with $0 < b < \min\{2,N\}$, $N\geq1$, and $q>2$. Then, there exists $c_2^*>0$ which is given precisely by \eqref{4.1}, such that the functional $I(u)$ has no any critical points for all $c\leq c_2^*$. In particular, for any $c>0$, $m(c)$ has no minimizers.
\end{theorem}
\begin{remark}\label{rek1.2}
In Lemma \ref{lm4.1}, we show that 
\begin{equation}\label{I-4.2}
  m(c)=\begin {cases}
  0 , & 0<c \leq c_2^{*};\\
  -\infty, & c>c_2^{*}.
  \end{cases}
\end{equation}
Though in this case for any $c>0$, $m(c)$ has no minimizers, we can observe from \eqref{4.4} that when $c>c_2^*$, there is a mountain pass geometry of $I(u)$ on $S(c)$, then one could expect in this case the existence of a mountain-pass type of solution.
\end{remark}

In what follows, we consider the mass-supercritical case, $p>p_q^*$, where as mentioned above the functional is unbounded from below, then we had to turn to consider other type of critical points of $I(u)$ on $S(c)$. Inspired by the work of \cite{BJL, J, LH}, we consider the following local type minimization problem:
\begin{equation}\label{I-localmini}
\gamma(c) := \inf_{u \in V(c)} I(u),\ c>0,
\end{equation}
where $V(c):=\{u\in S(c)\ :\ P(u)=0\}$, with
\begin{equation}\label{I-Q}
P(u) := \| \nabla u\|_2^2 + \frac{N(q-2)+2q}{2q} \| \nabla u\|_q^q - \frac{N(p-2)+2b}{2p} \int\frac{|u|^p}{|x|^b}dx.
\end{equation}
In Lemma \ref{lm5.0}, we show that if $u$ is a weak solution of Eq. \eqref{eq1.1}, then necessarily $P(u)=0$. Therefore, ``$P(u)=0$" is referred as a variant of Pohozaev identity. In addition, from Lemma \ref{lm5.1} and Lemma \ref{lm5.2}, we see that for all $c>0$, $V(c)\neq \emptyset$, hence $\gamma(c)$ is well-defined on $\R^+$.
\begin{theorem}[Mass supercritical case]\label{I-th-supercritical}
Assume that $p_q^*<p<\frac{q(N-b)}{(N-q)^+}$. Then $\gamma(c)$ admits a minimizer, provided that $p,q$ and $c$ satisfy one of the following conditions:
\begin{itemize}
  \item [(1)] If $N=1,2$, $p_q^* < p <+\infty$, and $c >0$;
  \item [(2)] If $N\geq 3$ with $q<\frac{2(N^2-2b)}{N^2-4}$, $p_q^* < p < 2_b^*:=\frac{2(N-b)}{(N-2)}$, and $c>0$.
\end{itemize}
In addition, any minimizer of $\gamma(c)$ can be assumed to be non-negative, radially symmetric and radially decreasing with respect to some point. Moreover, for a minimizer $u_c$ of $\gamma(c)$, there exists $\lambda_c<0$, such that $(\lambda_c, u_c)$ solves Eq. \eqref{eq1.1}, and 
\begin{align}\label{I-r5.1}
I(u_c)\to +\infty,\ \mbox{and }\ \lambda_c\to -\infty,  \ \mbox{as }\ c\to 0^+.
\end{align} 
\begin{align}\label{I-r5.2}
I(u_c)\to 0,\ \mbox{and }\ \lambda_c\to 0^-,  \ \mbox{as }\ c\to \infty.
\end{align}
\end{theorem}

\begin{remark}\label{rek1.3} The main idea of the proof is to show the compactness of a minimizing sequence of $\gamma(c)$. However, compared with the case of $m(c)$ in Theorem \ref{I-th-subcritical}, here it is more delicate since there additional information $P(u)=0$ is involved. To show that the limit function of a minimizing sequence (up to a subsequence if necessary), locates on $S(c)$, one needs to prove the strictly decreasing of the mapping: $c\mapsto \gamma(c)$. To this aim, we prove in Lemma \ref{lm5.8} that if its corresponding Lagrange multiplier is negative, then the strictly decreasing holds. It is where the conditions $(1) (2)$ of Theorem \ref{I-th-supercritical} are proposed. See Lemma \ref{lm5.5} for the details. 
\end{remark}

\begin{remark}\label{rek1.4}
Let us note that, when $N\geq 3$, $p_q^*<2_b^*$ if and only if $2<q<\frac{2(N^2-2b)}{N^2-4}$. When $2_q^*\leq p<\frac{q(N-b)}{(N-q)^+}$, we conjecture that $\gamma(c)$ admits a minimizer as $c>0$ small. However, a rigours proof is still open for us. In Remark \ref{rek5.1}, we provide a part proof that the corresponding Lagrange multiplier is nagetive as $c>0$ small, then the existence of minimizers follows. 
To improve further the condition of $(2)$, one may need to know whether the uniqueness or regularity of positive solutions to the zero-mass problem of Eq. \eqref{eq1.1} (i.e. $\lambda=0$) holds or not. Once we know that the positive solution of the zero-mass problem of Eq. \eqref{eq1.1} is unique, or solutions do not belong to $L^2(\R^N)$, then the conditions on $p,c$ can be given sharply for the existence of minimizers of $\gamma(c)$. In this direction, there have already been some trying. We refer the readers to the recent work \cite{BFG} and \cite{JZZ} for other models. However, since here mixing a Laplacian operator and $p$-Laplaican one with different rates, the situation is more complicated than the one in \cite{BFG, JZZ}, we might as well leave this interesting topic to a forthcoming paper.
\end{remark}

As a byproduct of some elements developed in the mass-supercritical case, inspired mainly by the work of Bartsch et.al \cite{BV}, we use the fountain theorem to prove the existence of infinitely many bound state normalized solutions of Eq. \eqref{eq1.1}. 
\begin{theorem}\label{I-th-multiplicity}
Assume that $p_q^*<p<\frac{q(N-b)}{(N-q)^+}$, and $p,q,c$ satisfy one of the following conditions:
\begin{itemize}
  \item [(1)] If $N=1,2$, $p_q^* < p <+\infty$, and $c >0$;
  \item [(2)] If $N\geq 3$ with $q<\frac{2(N^2-2b)}{N^2-4}$, $p_q^* < p < 2_b^*:=\frac{2(N-b)}{(N-2)}$, and $c>0$.
\end{itemize}
Then, there exists a sequence of couple solutions $\{(\lambda_n,u_n)\}$ of Eq. \eqref{eq1.1}, with $\|u_n\|_2^2=c$ and $\lambda_n<0$. Moreover, $I(u_n)\to +\infty$ as $n\to \infty$.
\end{theorem}
\begin{remark}\label{rek1.5}
Similar results have been established for the Schr\"odinger-Poisson-Slater equations in \cite{Luo}, for a class of nonlinear Choquard equations in \cite{BLL}, and recently for a type of $(2,q)$-Laplacian Equation with homogenous nonlinearity in \cite{BY}. However, compared with the mentioned references, in our proof, we use some newly developed techniques proposed by Cingolani and Tanaka \cite{CT}, which help to simply the proof. 
\end{remark}

This paper is structured as follows. In Section \ref{sec-GN}, we rigorously establish a sharp inhomogeneous $L^{q}$-Gagliardo-Nirenberg inequality, which will crucially used in the proofs of the main results, also we prove a Pohozaev identity. We handle the mass subcritical case and mass critical case respectively in Section \ref{sec-subcritical}, and Section \ref{sec-critical}, where we give particularly the properties of $m(c)$. In Section \ref{sec-supercritical}, we prove the existence of minimizers of $\gamma(c)$ and a multiplicity of bounded states of normalized solutions. The proofs of the main results are carried out in Section \ref{sec-proofs}. Finally, we establish in Section \ref{sec-apt} several technique lemmas, which are necessarily used in our proofs. 

\section{Some Preliminaries}\label{sec-GN}
In this section, we give some preliminaries. Inspired by the work of Genoud \cite{G}, we first prove the following sharp inhomogeneous type of $L^q$- Gagliardo-Nirenberg inequality, which is interesting by itself but also essential for various estimates in our proofs. 
\begin{lemma}[Sharp inhomogeneous $L^{q}$-Gagliardo-Nirenberg inequality]\label{th-main}
Assume that $2<p<\frac{q(N-b)}{(N-q)^+}$, with $0 < b < \min\{2,N\}$, $N\geq1$, $q\geq 2 $, then  there exists a sharp constant $\mathcal{K}_{N,p,q}:=\frac{p}{\|Q_{p,q}\|_2^{p-2}}>0$, such that
\\
\begin{equation}\label{p-GN}
  \int \frac{|u|^{p}}{|x|^{b}} dx \leq \mathcal{K}_{N,p,q} \|\nabla u\|_q^{\sigma_{p,q}}\|u\|_2^{p-\sigma_{p,q}}, \quad \forall u \in D^{1,q}(\R^N)\cap L^{2}(\mathbb{R}^{N}),
\end{equation}
where $\sigma_{p,q}:=\frac{q[N(p-2)+2b]}{N(q-2)+2q}$, and $Q_{p,q}$ is a fixed ground state solution of
\begin{equation}\label{eq2.1}
\sigma_{p,q} \triangle_q Q + (p-\sigma_{p,q}) Q - |x|^{-b}|Q|^{p-2}Q=0.
\end{equation}
Moreover, the equality in \eqref{p-GN} holds if $u=Q_{p,q}$. In particular, let $q=2$, then \eqref{p-GN} is reduced to 
\begin{equation}\label{p-GN2}
  \int \frac{|u|^{p}}{|x|^{b}} dx \leq \mathcal{K}_{N,p,2} \|\nabla u\|_2^{\frac{N(p-2)+2b}{2}}\|u\|_2^{\frac{2N+2b-(N-2)p}{2}}, \quad \forall u \in H^1(\R^N).
\end{equation}
\end{lemma}

\begin{proof}
To prove \eqref{p-GN}, we define the following Weinstein-type functional:\\
$$ J(u): = \frac{\|\nabla u\|_q^{\sigma_{p,q}} \|u\|_2^{p-\sigma_{p,q}}}{\int|x|^{-b}|u|^{p}dx} , \quad \forall u\in D^{1,q}(\R^N)\cap L^{2}(\mathbb{R}^N) \backslash \{0\}.$$

Note that $p-\sigma_{p,q}=\frac{2[(N-b)q-(N-q)p]}{(N+2)q-2N}>0$ if $p<\frac{q(N-b)}{(N-q)^+}$. Hence, it is easy to verify that $J(u)$ is of class $C^1$, and $J(u)\geq 0$ in $D^{1,q}(\R^N)\cap L^{2}(\mathbb{R}^N) \backslash \{0\}$ . 

Define 
\begin{equation}
m:= \inf_{u \in D^{1,q}(\R^N)\cap L^{2}(\mathbb{R}^N) \backslash \{0\}} J(u),
\end{equation}
then clearly $m\geq 0$. We shall show that $m>0$ is attained by some non-trivial function.  As a first step, let $\{u_n\}$ be a minimizing sequence of $m$.  Since $J(|u|)\leq J(u)$, hence $u_n$ can be assumed to be non-negative. Thus we can denote by $\{u_n^{*}\}$ the Schwartz-symmetrization of $\{u_n\}$, and by the Polya-Szeg\"o inequality and Lemma \ref{lm_ap2}, we have $J(u_n^{*})\leq J(u_n)$. This means that  $\{u_n^*\}$ is also a minimizing sequence of $m$. Without loss of generality, we can assume that $\{u_n\}$ is non-negative, radial and radially decreasing.\\

Now for each $n \in \mathbb{N}^{+}$, we set
$$v_n:= \lambda_n u_n(\eta_n x),$$
with
$$\lambda_n := \frac{\|u_n\|_2^{\frac{2(N-q)}{(N+2)q-2N}}}{\|\nabla u_n\|_q^{\frac{Nq}{(N+2)q-2N}}},\quad \eta_n:=(\frac{\|u_n\|_2}{\|\nabla u_n\|_q})^{\frac{2q}{(N+2)q-2N}}.$$
Then $\|v_n\|_2=\|\nabla v_n\|_q=1$ and $J(v_n)=J(u_n)$. Clearly, $\{v_n\}$ is bounded in $D^{1,q}(\R^N)\cap L^2(\mathbb{R}^{N})$ , then there exists $v^* \in D^{1,q}(\R^N)\cap L^2(\mathbb{R}^{N}) $ with $v^*\neq 0$, such that $v_n\rightharpoonup v^{*}$  in  $D^{1,q}(\R^N)\cap L^2(\mathbb{R}^{N})$, and by Lemma \ref{lm_ap1}, $\int |x|^{-b}|v_n|^pdx\rightarrow \int |x|^{-b}|v^*|^pdx$. Thus,
$$ m \leq J(v^*)\leq\frac{1}{I(v^*)}=\lim_{n \rightarrow +\infty} J(v_n)=m, $$
where $I(\cdot) := \int |x|^{-b}|\cdot|^pdx$. From this, we deduce that
\begin{equation}\label{2.00}
m=J(v^*)=\frac{1}{I(v^*)}=\Big(\int |x|^{-b}|v^*|^pdx\Big)^{-1},\ \|v^*\|_2=\|\nabla v^*\|_q=1.
\end{equation}
Hence, $ v_n\rightarrow v^*$ in $D^{1,q} \cap L^2(\mathbb{R}^{N})$. In particular, $v^*$ is a minimizer of $m$.  Furthermore, $v^*$ is a critical point of $J(u)$, then we have $$\frac{d}{dt}J(v^*+t\varphi)\mid_{t=0}=0,\quad \forall \varphi \in C_0^{\infty}(\mathbb{R}^{N}).$$

Due to \eqref{2.00}, this implies that 
\begin{equation*}
  \sigma_{p,q} \int |x|^{-b}|v^*|^pdx \langle \triangle_q v^*,\varphi \rangle+(p-\sigma_{p,q})\int|x|^{-b}|v^*|^pdx \langle  v^*,\varphi \rangle-p \langle  |x|^{-b}|v^*|^{p-2}v^*,\varphi \rangle =0,
\end{equation*}
or equivalently 
\begin{equation}\label{2.01}
  \sigma_{p,q} \triangle_q v^* + (p-\sigma_{p,q})v^* - mp|x|^{-b}|v^*|^{p-2}v^*=0.
\end{equation}

Taking $\omega:=(mp)^{\frac{1}{p-2}}v^*$, then $\omega$ satisfies
\begin{equation}\label{2.02}
\sigma_{p,q} \triangle_q \omega + (p-\sigma_{p,q})\omega - |x|^{-b}|\omega|^{p-2}\omega=0.
\end{equation}

Note from \eqref{2.00} that 
\begin{equation*}
 m=\frac{1}{\int|x|^{-b}|v^*|^pdx}=\frac{1}{(mp)^{\frac{-p}{p-2}}\int |x|^{-b}|\omega|^pdx},
\end{equation*}
which implies that 
\begin{equation}\label{2.03}
m=p^{-\frac{p}{2}}\Big(\int |x|^{-b}|\omega|^pdx\Big)^{\frac{p-2}{2}}.
\end{equation}
Next, we shall prove 
\begin{equation}\label{2.04}
\int |x|^{-b}|\omega|^pdx=p\|w\|_2^2.
\end{equation} 
Due to \eqref{2.03} and \eqref{2.04}, then $m=p^{-1}\|\omega\|_2^{p-2}$, thus letting $\mathcal{K}_{N,p,q}:= m^{-1}=\frac{p}{\|\omega\|_2^{p-2}}$, we then obtain \eqref{p-GN}. 

To end the proof, we only need to verify \eqref{2.04}. Multiplying the equation \eqref{2.02} by $\omega$ and integrating by parts, we have
\begin{align}\label{2.05}
 \sigma_{p,q}\|\nabla \omega\|_q^q+(p-\sigma_{p,q})\|\omega\|_2^2=\int |x|^{-b}|\omega|^pdx.
\end{align}
Let
\begin{align*}
 E(\omega):=\frac{\sigma_{p,q}}{q}\|\nabla \omega\|_q^q+\frac{(p-\sigma_{p,q})}{2}\|\omega\|_2^2-\frac{1}{p}\int |x|^{-b}|\omega|^pdx,
\end{align*}
and set $\omega_t:=t^{\frac{N}{2}}\omega(tx)$, then 
\begin{equation*}
E(\omega_t)=\frac{\sigma_{p,q}}{q} t^{\frac{N(q-2)+2q}{2}}\|\nabla \omega\|_q^q+\frac{(p-\sigma_{p,q})}{2}\|\omega\|_2^2-\frac{t^{\frac{N(p-2)+2b}{2}}}{p}\int |x|^{-b}|\omega|^pdx.
\end{equation*}
Since $\omega$ is a solution of \eqref{2.02}, then $\frac{d}{dt}E(\omega_t)|_{t=1}=0$. By direct calculation, we have
$$ \frac{\sigma_{p,q}}{q} \|\nabla \omega\|_q^q = \frac{N(p-2)+2b}{p[(N+2)q-2N]}\int |x|^{-b}|\omega|^pdx=\frac{\sigma_{p,q}}{pq}\int |x|^{-b}|\omega|^pdx.$$
Therefore, $\|\nabla \omega\|_q^q=\frac{1}{p}\int |x|^{-b}|\omega|^p$. Taking it into \eqref{2.05}, we have
\begin{equation*}
 \frac{p-\sigma_{p,q}}{p}\int |x|^{-b}|\omega|^pdx=(p-\sigma_{p,q})\|\omega\|_2^2.
\end{equation*}
Since $p-\sigma_{p,q}>0$ if $p<\frac{q(N-b)}{(N-q)^+}$, then \eqref{2.04} follows. The proof is completed.
\end{proof}

Next, we prove the following Pohozaev identity.
\begin{lemma}\label{lm5.0}
Assume that $2<p<\frac{q(N-b)}{(N-q)^+}$. If $u_0\in X$ is a weak solution of the following equation:
 \begin{equation}\label{eq5.0}
 -\triangle u -\triangle_q u =\lambda u + \frac{|u|^{p-2}u}{|x|^b}, \quad x\in \mathbb{R}^N.
\end{equation}
Then $P(u_0)=0$, where $P(u)$ is given in \eqref{I-Q}. 
\end{lemma}
\begin{proof}
If $u_0\in X$ is a weak solution of Eq. \eqref{eq5.0}, then multiplying \eqref{eq5.0} by $u_0$ and $x\cdot \nabla u_0$ respectively, and integrating by part, we have
\begin{equation}\label{5.1}
\begin{split}
\|\nabla u_0 \|_2^2+\|\nabla u_0\|_q^q-\int \frac{|u_0|^{p}}{|x|^b}dx -\lambda \|u_0\|_2^2 = 0,
\end{split}
\end{equation}

\begin{equation}\label{5.2}
\begin{split}
 \frac{-(N-2)}{2}\|\nabla u_0 \|_2^2 -\frac{N-q}{q}\|\nabla u_0\|_q^q -\frac{b-N}{p} \int \frac{|u_0|^{p}}{|x|^b}dx + \frac{N\lambda}{2} \|u_0\|_2^2 = 0.
\end{split}
\end{equation}
From \eqref{5.1} and \eqref{5.2}, we have 
$$\| \nabla u_0\|_2^2 + \frac{N(q-2)+2q}{2q} \| \nabla u_0\|_q^q - \frac{N(p-2)+2b}{2p} \int\frac{|u_0|^p}{|x|^b}dx = 0,$$
which is exactly $P(u_0)=0$.\\
\end{proof}


\section{THE MASS SUBCRITICAL CASE}\label{sec-subcritical}

In this section, we deal with the existence of minimizers to $m(c)$ in the mass subcritical. As a first step, we prove the properties of $m(c)$, which is essential to establish a sharp existence of minimizers of $m(c)$, which depends on different assumptions on the parameter $p$ and the mass $c>0$. Then, we use the concentration compactness principle to show the existence of minimizers of $m(c)$.

\begin{lemma}\label{lm3.1}
Assume that $2<p<p_q^*:=\frac{2(q-b)}{N}+q$, with $0 < b < \min\{2,N\}$, $N\geq1$, and $q>2$. Then, for any $c>0$, there holds that 
$$-\infty<m(c)\leq 0.$$
Moreover, define
\begin{equation}\label{c1}
c_1^*:=\inf\{c>0\mid m(c) < 0\},
\end{equation}
then 
\begin{itemize}
  \item [(1)] If $2<p<\frac{2(2-b)}{N}+2$, then $-\infty<m(c)<0$, for all $c>0$;
  \item [(2)] If $\frac{2(2-b)}{N}+2\leq p <p_q^*$, then $0<c_1^*<+\infty$, and 
\begin{eqnarray*}
\left\{
\begin{array}{l}
m(c)=0, \qquad \qquad 0<c\leq c_1^*;\\
-\infty<m(c)<0,\quad c>c_1^*.
\end{array}
\right.
\end{eqnarray*}
  \item [(3)] The mapping $c\mapsto m(c)$ is continuous and non-increasing on $\R^+$;
  \item [(4)] If $m(c)$ is attained by some $u_c\in S(c)$, then for all $\bar{c}>c$, $m(\bar{c})<m(c)$.
\end{itemize}
\end{lemma}
\begin{proof}[Proof of Lemma \ref{lm3.1}]

Given $c>0$, for any $u\in S(c)$, using the $L^{q}$-Gagliardo-Nirenberg inequality \eqref{p-GN}, we have
\begin{equation}\label{3.2}
  \begin{split}
  I(u)&=\frac{1}{2}\|\nabla u\|_2^{2}+\frac{1}{q}\|\nabla u\|_q^{q}-\frac{1}{p} \int\frac{|u|^{p}}{|x|^{b}}dx\\
      &\geq \frac{1}{2}\|\nabla u\|_2^{2}+\frac{1}{q}\|\nabla u\|_q^{q}-\frac{1}{p} \mathcal{K}_{N,p,q} \|\nabla u\|_q^{\sigma_{p,q}}\|u\|_2^{p-\sigma_{p,q}}\\
      &\geq \frac{1}{q}\|\nabla u\|_q^{q}-\frac{1}{p} \mathcal{K}_{N,p,q} \|\nabla u\|_q^{\sigma_{p,q}}c^{\frac{p-\sigma_{p,q}}{2}}.
  \end{split}
\end{equation}
Since $q>\sigma_{p,q}$ if $p<\frac{2(q-b)}{N}+q$, then from \eqref{3.2} we deduce that $m(c)>-\infty$. 

In addition, the proof of the continuity and non-increasing of $m(c)$ is somehow standard, here we only refer the readers to the reference such as \cite[Theorem 1.2 (4)]{BHHL}. 

Let $u\in S(c)$ be fixed, and denote $u_t(x):=t^{\frac{N}{2}}u(tx)\in S(c)$, then $u_t\in S(c)$ for all $t>0$, and
\begin{equation}\label{scaling1}
I(u_t)=\frac{t^{2}}{2}\|\nabla u\|_2^{2}+\frac{t^{\frac{N(q-2)+2q}{2}}}{q}\|\nabla u\|_q^{q}-\frac{t^{\frac{N(p-2)+2b}{2}}}{p} \int\frac{|u|^{p}}{|x|^{b}}dx.
\end{equation}
This shows that $\lim_{t\to 0^+}I(u_t)=0$ and then $m(c)\leq 0$. Then we conclude that if $p<\frac{2(q-b)}{N}+q$, then $-\infty <m(c)\leq0$, for all $c>0$.

When $2<p<\frac{2(2-b)}{N}+2$, then $\frac{N(p-2)+2b}{2}<\min\{2, \frac{N(q-2)+2q}{2}\}$, thus we observe from \eqref{scaling1} that $I(u_t)<0$ as $t>0$ small enough. This implies that $m(c)<0$ for all $c>0$, and then the point $(1)$ is verified.
 
When $\frac{2(2-b)}{N}+2\leq p <p_q^*$, by the continuity and non-increasing of $m(c)$, and also the definition of $c_1^*$, to verify the point $(2)$ it suffices to prove that $0<c_1^{*}<+\infty$.

We first show that $c_1^{*}\neq 0$. Indeed, let $p_2^{*}=\frac{2(2-b)}{N}+2$ and $p_q^{*}=\frac{2(q-b)}{N}+q$, since $p_2^*\leq p<p_q^*$, then there exist constants $c_1>0,c_2>0$, such that
\begin{equation}\label{3.4}
 \frac{1}{p}\int \frac{|u|^{p}}{|x|^{b}} dx \leq c_1\int \frac{|u|^{p_2^*}}{|x|^{b}} dx+c_2\int \frac{|u|^{p_q^*}}{|x|^{b}} dx.
\end{equation}
Then the $L^{q}$-Gagliardo-Nirenberg inequality \eqref{p-GN} yields to 
\begin{equation}\label{3.5}
 c_1\int \frac{|u|^{p_2^*}}{|x|^{b}} dx \leq c_1(p_2^*) \|\nabla u\|_2^{2}\|u\|_2^{p_2^*-2},
\end{equation}
\begin{equation}\label{3.6}
c_2\int \frac{|u|^{p_q^*}}{|x|^{b}} dx \leq c_2(p_q^*) \|\nabla u\|_q^{q}\|u\|_2^{p_q^*-q}.
\end{equation}
From \eqref{3.4}, \eqref{3.5} and \eqref{3.6}, we have for all $u\in S(c)$ that
\begin{equation}\label{3.7}
  \begin{split}
    I(u)&=\frac{1}{2}\|\nabla u\|_2^{2}+\frac{1}{q}\|\nabla u\|_q^{q}-\frac{1}{p} \int\frac{|u|^{p}}{|x|^{b}}\\
    &\geq \Big[\frac{1}{2}- c_1(p_2^*) c^{\frac{p_q^*-2}{2}}\Big]\|\nabla u\|_2^{2}+\Big[\frac{1}{q}- c_2(p_2^*) c^{\frac{p_q^*-q}{2}}\Big]\|\nabla u\|_q^{q}.
  \end{split}
\end{equation}
Note that $p_2^*-2>0, p_q^*-q>0$, then \eqref{3.7} shows that $m(c)\geq 0$ if $c>0$ is sufficiently small. Since we have proved that $m(c) \leq 0$ for all $c>0$. Then we have $m(c)=0$ as $c>0$ small, which verifies that $c_1^{*}\neq 0$.

Next we show that $c_1^{*}\neq +\infty$. Note that choosing $u_0 \in X$ such that $\|u_0\|_2^2=1$, and let $u_0^{c}:=c^{\frac{1}{N+2}}u_0(c^{-\frac{1}{N+2}}x)$, then $u_0^c\in S(c)$, and
\begin{equation}
 I(u_0^{c}) = \frac{c^{\frac{N}{N+2}}}{2}\|\nabla u_0\|_2^{2} + \frac{c^{\frac{N}{N+2}}}{q}\|\nabla u_0\|_q^{q} - \frac{c^{\frac{p-b}{N+2}+\frac{N}{N+2}}}{p}\int \frac{|u_0|^{p}}{{|x|^{b}}} dx,
\end{equation}
from which, we see that $I(u_0^{c})<0$, as $c>0$ is sufficiently large, so when $c>0$ is large, $m(c)<0$. Thus by the definition of $c_1^*$, we have $c_1^{*}\neq +\infty $. 

Finally, to show the point $(4)$, by assumption, $m(c)=I(u_c)\leq 0$, then we observe that for $t:=\sqrt{\frac{\bar{c}}{c}}>1$, since $p>q>2$, then 
\begin{align}\label{3.7.0}
\begin{split}
I(tu_c)&=\frac{t^2}{2} \|\nabla u_c\|_2^{2} + \frac{t^q}{q}\|\nabla u_c\|_q^{q} - \frac{t^{p}}{p}\int \frac{|u_c|^{p}}{|x|^b}dx\\
&< t^q I(u_c)=t^q m(c)\leq m(c),
\end{split}
\end{align}
which implies that $m(\bar{c})<m(c)$. Then the proof is completed.
\end{proof}

\begin{remark}\label{rek3.1}
From \eqref{scaling1}, we can observe that if $p>p_q^*$, then for all $c>0$, given $u\in S(c)$, then $I(u_t)\to -\infty$ as $t\to \infty$, which implies that
$m(c)=-\infty$, or equivalently, when $p>p_q^*$, then for all $c>0$, the functional $I(u)$ is unbounded from below on $S(c)$. 
\end{remark}

Using the points $(2)$ and $(4)$ of Lemma \ref{lm3.1}, we obtain immediately that following non-existence.
\begin{corollary}\label{col3.1}
When $\frac{2(2-b)}{N}+2\leq p <p_q^*$, then for all $0<c<c_1^*$, $m(c)$ has no minimizers.  
\end{corollary}
\begin{proof}
Indeed, if we assume that there exists $c'<c_1^*$, such that $m(c')$ admits a minimizer $u_{c'}\in S(c')$, then using $(4)$ of Lemma \ref{lm3.1}, we have for any $c$ with $c'<c<c_1^*$, $m(c)<m(c')=0$, this contradicts with the point $(2)$. Then the non-existence has been proved.
\end{proof}

\begin{proposition}\label{prop2.1} 
Assume that $2<p<p_q^*$, with $0<b<\min\{2,N\}$, $N\geq1$, and $q>2$. Suppose that for some $c>0$, $m(c)$ satisfies that
\begin{eqnarray}
-\infty<m(c)<0.
\end{eqnarray}
Then any minimizing sequence of $m(c)$ is pre-compact in $X$. In particular, $m(c)$ admits a minimizer $u_c\in S(c)$. Moreover, $u_c\in S(c)$ can be assumed to be non-negative, radially symmetric and radially decreasing with respect to some point. 
\end{proposition}

\begin{proof}[Proof of Proposition \ref{prop2.1}] Let $\{u_n\}$ be a minimizing sequence of $m(c)$, namely
$$\|u_n\|_2^2=c \ \mbox{ and }\  I(u_n)\to m(c).$$

First, we show that $\{u_n\}$ is bounded in $X$. Indeed, as \eqref{3.2}, using \eqref{p-GN} we have
\begin{equation}\label{3.7.1}
m(c)+o_n(1)=I(u_n)\geq \frac{1}{2}\|\nabla u_n\|_2^{2}+\frac{1}{q}\|\nabla u_n\|_q^{q}-\frac{1}{p} \mathcal{K}_{N,p,q} \|\nabla u_n\|_q^{\sigma_{p,q}}c^{\frac{p-\sigma_{p,q}}{2}}.
\end{equation}
This proves that $\{\|\nabla u_n\|_q^{q}\}$ is bounded, since $q>\sigma_{p,q}$. Using \eqref{p-GN2} ($q=2$) and the same arguments, we obtain that $\{\|\nabla u_n\|_2^{2}\}$ is bounded. Hence, $\{u_n\}$ is bounded in $X$. 

Thus, up to a subsequence, there exists $u\in X$, such that
$$u_n \rightharpoonup u  \mbox{ in } X, \ \mbox{ and }\  u_n \to u \ \mbox{in}\ L^r_{loc}(\mathbb{R}^N), \forall r\in [2,\frac{2N}{(N-2)^+}).$$
In particular, by Lemma \ref{lm_ap1}, 
\begin{equation}\label{3.8}
\lim_{n\to \infty}\int \frac{|u_n|^p}{|x|^b}dx=\int \frac{|u|^p}{|x|^b}dx. 
\end{equation}
Then clearly $u\neq 0$, since if $u=0$, then as $n$ large enough, $m(c)=I(u_n)+o_n(1)\geq 0$, which contradicts with the assumption $m(c)<0$. 

Now, we prove that $u$ is a minimizer of $m(c)$. Indeed, since $u\neq 0$, by the Brezis-Lieb Lemma and Lemma \ref{lm_ap1}, we have
\begin{align}\label{3.9}
c=\|u_n\|_2^2=\|u_n-u\|_2^2+\|u\|_2^2+o_n(1).
\end{align}
\begin{align}\label{3.10}
I(u_n)=I(u_n-u)+I(u)+o_n(1).
\end{align}

If $\|u\|_2^2<c$, then we denote $u^{t_0}:=t_0u(t_0^{-1}x)$ with $t_0=\Big(\frac{c}{\|u\|_2^2}\Big)^{\frac{1}{N+2}}$, thus $t_0>1, u^{t_0}\in S(c)$ and
\begin{align}\label{3.11}
I(u^{t_0})= \frac{t_0^N}{2}\|\nabla u\|_2^2+\frac{t_0^N}{q}\|\nabla u\|_q^q - \frac{t_0^{p-b+N}}{p}\int \frac{|u|^p}{|x|^b}dx,
\end{align}
which implies that
\begin{align}\label{3.12}
I(u)= \frac{1}{t_0^N}I(u^{t_0}) +\frac{(t_0^{p-b}-1)}{p}\int \frac{|u|^p}{|x|^b}dx.
\end{align}
Similarly, let $t_n=\Big(\frac{c}{\|u_n-u\|_2^2}\Big)^{\frac{1}{N+2}}$, thus $t_n>1, (u_n-u)^{t_n}\in S(c)$ and
\begin{align}\label{3.13}
I(u_n-u)= \frac{1}{t_n^N}I((u_n-u)^{t_n}) +\frac{(t_n^{p-b}-1)}{p}\int \frac{|u_n-u|^p}{|x|^b}dx.
\end{align}
Therefore, by \eqref{3.9} and \eqref{3.10} we have
\begin{align}\label{3.14}
m(c)=&I(u_n)+o_n(1)=I(u_n-u)+I(u)+o_n(1)\nonumber\\
\geq & m(c)\left(\frac{1}{t_0^N}+\frac{1}{t_n^N}\right)+\frac{(t_0^{p-b}-1)}{p}\int \frac{|u|^p}{|x|^b}dx+\frac{(t_n^{p-b}-1)}{p}\int \frac{|u_n-u|^p}{|x|^b}dx \nonumber\\
\geq & m(c)+\frac{(t_0^{p-b}-1)}{p}\int \frac{|u|^p}{|x|^b}dx+o_n(1).
\end{align}
Since $p>b$ and $t_0>1$, then \eqref{3.14} is impossible. Hence, $\|u\|_2^2=c$, namely $u_n\to u$ in $L^2(\R^N)$. Then using \eqref{3.8} and the lower semi-continuity of norm, we have
\begin{align}\label{3.15}
m(c)\leq I(u)\leq \lim_{n\rightarrow \infty}I(u_{n})=m(c).
\end{align}
This implies that $I(u)=m(c)$ and  $\|u_{n}\|_{X}\rightarrow \|u\|_{X}$, then $u_n\rightarrow u$ in $X$, and $u$ is a minimizer of $m(c)$. Furthermore, we observe that $|u_c|$ is also a minimizer, hence without loss of generality, we can assume that $u_c\geq 0$.

Let $u_c^*$ be the Schwartz symmetrization of $u_c$, then by the Polya-Szeg\"o inequality (see e.g. \cite{LL}) and Lemma \ref{lm_ap2}, we have that
\begin{align*}
\|u_c^*\|_2^2=\|u_c\|_2^2=c,\ \|\nabla u_c^*\|_2^2 \leq \|\nabla  u_c\|_2^2,\ \|\nabla u_c^*\|_q^q\leq \|\nabla u_c\|_q^q,\ \int \frac{|u_c|^p}{|x|^b}dx \leq \int \frac{|u_c^*|^p}{|x|^b}dx.
\end{align*}
This implies that
\begin{align*}
m(c)\leq I(u_c^*)\leq I(u_c)=m(c).
\end{align*}
Hence, $u_c^*$ is also a minimizer of $m(c)$. Moreover,
\begin{eqnarray}\label{3.19}
\int\frac{|u_c|^p}{|x|^b}dx =\int \frac{|u_c^*|^p}{|x|^b}dx.
\end{eqnarray}
Thus, using Lemma \ref{lm_ap2}, we deduce from \eqref{3.19} that $u_c=u_c^*$. Hence $u_c$ is radially symmetric and decreasing with respect to some point.  At this point, the proposition is proved.
\end{proof}

We deal with the critical case $c=c_1^*$ in a single lemma.
\begin{lemma}\label{lm3.2}Assume that $\frac{2(2-b)}{N}+2\leq p<p_q^*$, then 
\begin{itemize}
  \item [(1)] When $\frac{2(2-b)}{N}+2<p<p_q^*$, then $m(c_1^*)$ admits a minimizer;
  \item [(2)] When $p=\frac{2(2-b)}{N}+2$, then $m(c_1^*)$ has no minimizers.
\end{itemize}
\end{lemma}
\begin{proof}
We use an approximation to derive the existence. Let $c_k:=c_1^*+\frac{1}{k}$, then for all $k$, $c_k>c_1^*$ and $m(c_k)$ admits a minimizer $u_k\in S(c)$, and $m(c_k)<0$. Now we want to show that when $\frac{2(2-b)}{N}+2< p <p_q^*$, then 
\begin{eqnarray}\label{k1}
u_k\to u_c \ \mbox{ in } X.
\end{eqnarray}
Indeed, if \eqref{k1} holds, then by the continuity of $m(c)$ in Lemma \ref{lm3.1}, $u_c\in S(c)$ is a minimizer of $m(c_1^*)=0$. Now we verify \eqref{k1}.  Note that 
$$I(u_k)=m(c_1^*)+o_n(1)=o_n(1),$$ 
then as in Proposition \ref{prop2.1}, by \eqref{3.7.1}, $\{u_k\}$ is bounded in $X$. 

Thus, up to a subsequence, there exists $u_c\in X$, such that
$$u_n \rightharpoonup u_c  \mbox{ in } X, \ \mbox{ and }\  u_n \to u_c \ \mbox{in}\ L^r_{loc}(\mathbb{R}^N), \forall r\in [2,\frac{2N}{(N-2)^+}).$$
By Lemma \ref{lm_ap1}, 
\begin{equation}\label{3.20}
\lim_{n\to \infty}\int \frac{|u_n|^p}{|x|^b}dx=\int \frac{|u_c|^p}{|x|^b}dx. 
\end{equation}
Then we claim that $u_c\neq 0$. Indeed, if $u_c=0$, then as $n$ large enough, 
$$o_n(1)=I(u_k)=\frac{1}{2}\|\nabla u_k\|_2^2+\|\nabla u_k\|_q^q+o_n(1),$$ 
which implies that $\|\nabla u_k\|_2^2=o_n(1)$. However, using the $L^{q}$-Gagliardo-Nirenberg inequality \eqref{p-GN2}, we have 
$$o_n(1)=I(u_k)\geq \frac{1}{2}\|\nabla u_k\|_2^{2}-\frac{1}{p} \mathcal{K}_{N,p,2}\|\nabla u_k\|_2^{\frac{N(p-2)+2b}{2}}(c_1^*)^{\frac{2N+2b-(N-2)p}{4}}.$$
This is impossible, since $\frac{N(p-2)+2b}{2}>2$ as $p>\frac{2(2-b)}{N}+2$ and $\|\nabla u_k\|_2^2=o_n(1)$. Then the claim follows.

Having proved the boundedness and non-vanishing of $\{u_k\}$ with $I(u_k)=m(c_k)$, we then using the same compactness arguments as in Proposition \ref{prop2.1}, get \eqref{k1}.\\

Now, when $p=\frac{2(2-b)}{N}+2$, if we assume by contradiction that $m(c_1^*)$ admits a minimizer $u_1\in S(c)$, then clearly $u_1$ is a weak solution of Eq. \eqref{eq1.1}, and by Lemma \ref{lm5.0}, $P(u_1)=0$. Thus we have
\begin{align*}
  0=m(c_1^*)=I(u_1)=I(u_1)-\frac{2}{N(p-2)+2b}P(u_1)=\dfrac{(N+2)(2-q)}{4q}\| \nabla u_1\|_q^q.
\end{align*}
This is a contradiction, since $u_1\in S(c)$. Then we have proved that when $p=\frac{2(2-b)}{N}+2$, $m(c_1^*)$ has no minimizers.  \\
\end{proof}


\section{THE MASS CRITICAL CASE}\label{sec-subcritical}

In this section, we consider the mass critical case $p_q^*:=\frac{2(q-b)}{N}+q$. Recall that when $p=p_q^*$, the $L^{q}$-Gagliardo-Nirenberg inequality \eqref{p-GN} is reduced to 
\begin{equation}\label{p-GN3}
  \int \frac{|u|^{p_q^*}}{|x|^{b}} dx \leq \mathcal{K}_{N,q}^* \|\nabla u\|_q^{q}\|u\|_2^{p_q^*-q}, \quad \forall u \in D^{1,q}(\R^N)\cap L^{2}(\mathbb{R}^{N}),
\end{equation}
where $\mathcal{K}_{N,q}^*:=\frac{p_q^*}{\|Q_q^*\|_2^{p_q^*-2}}$, and $Q_q^*$ is a fixed ground state solution of
\begin{equation}\label{eq4.1}
q\triangle_q Q + (p_q^*-q) Q - |x|^{-b}|Q|^{p_q^*-2}Q=0.
\end{equation}
Moreover, the equality in \eqref{p-GN3} holds if $u=Q_q^*$.

\begin{lemma}\label{lm4.1}
Let $0<b<\min\{2,N\}$, $N\geq1$, and $q>2$, $p=p_q^*$, and denote
\begin{equation}\label{4.1}
c_2^*:=\Big(\frac{p_q^*}{q\mathcal{K}_{N,q}^*}\Big)^{\frac{2}{p_q^*-q}}=q^{\frac{2}{q-p_q^*}}\|Q_q^*\|_2^{\frac{2(p_q^*-2)}{p_q^*-q}}.
\end{equation}
Then, 
\begin{equation}\label{4.2}
  m(c)=\begin {cases}
  0 , & 0<c \leq c_2^{*};\\
  -\infty, & c>c_2^{*}.
  \end{cases}
\end{equation}
Moreover, the functional $I(u)$ has no any critical points for all $c\leq c_2^*$. In particular, $m(c)$ has no minimizers for all $c>0$.
\end{lemma}

\begin{proof}[Proof of Lemma 4.1]
For any $u\in S(c)$, using the Gagliardo-Nirenberg inequality \eqref{p-GN3}, we have
\begin{equation}\label{4.3}
\begin{split}
 I(u)  &\geq \frac{1}{2}\|\nabla u\|_2^{2}+\frac{1}{q}\|\nabla u\|_q^{q}-\frac{1}{p_q^*} \mathcal{K}_{N,q}^* \|\nabla u\|_q^{q}\|u\|_2^{p_q^*-q}\\
       & = \frac{1}{2}\|\nabla u\|_2^{2}+ (\frac{1}{q}-\frac{1}{p_q^*} \mathcal{K}_{N,q}^* c^{\frac{p_q^*-q}{2}})  \|\nabla u\|_q^{q}\\
       & = \frac{1}{2}\|\nabla u\|_2^{2}+\frac{1}{q}\Big[1-\Big(\frac{c}{c_2^*}\Big)^{\frac{p_q^*-q}{2}}\Big] \|\nabla u\|_q^{q}.
\end{split}
\end{equation}
This yields that when $0<c \leq c_2^*$, $I(u)\geq 0$ for any $u \in S(c)$. Then as in the proof of Lemma \ref{lm3.1}, we conclude that $m(c)=0$ as $0<c \leq c_2^*$. 

On the other hand, using the scaling in \eqref{3.11} with $u=Q_q^*, t_0=\Big(\frac{c}{\|Q_q^*\|_2^2}\Big)^{\frac{1}{N+2}}$ and $p=p_q^*$, and also the scaling in \eqref{scaling1}, then $(u^{t_0})_t\in S(c)$ for all $t>0$, and  
\begin{align}\label{4.4}
 \begin{split}
I((u^{t_0})_t) &=t_0^N\Big [\frac{t^{2}}{2}\|\nabla Q_q^*\|_2^{2}+\frac{t^{\frac{N(q-2)+2q}{2}}}{q}\|\nabla Q_q^*\|_q^{q}-\frac{t^{\frac{N(q-2)+2q}{2}}\cdot t_0^{p_q^*-b}}{p_q^*} \mathcal{K}_{N,q}^* \|\nabla Q_q^*\|_q^{q}\|Q_q^*\|_2^{p_q^*-q}\Big]\\
       &=t_0^N\Big\{\frac{t^{2}}{2}\|\nabla Q_q^*\|_2^{2} + \frac{t^{\frac{N(q-2)+2q}{2}}}{q}\Big[1-\Big(\frac{c}{c_2^*}\Big)^{\frac{q-b}{N}}\Big] \|\nabla Q_q^*\|_q^{q}\Big\}.
 \end{split}
\end{align}
This shows that when $c>c_2^*$, $\lim_{t\to +\infty}I((u^{t_0})_t) \to -\infty $, then  we have $m(c)=-\infty$ for $c>c_2^*$.  Thus \eqref{4.2} is verified.\\

Now we show that for $c\leq c_2^*$, $I(u)$ has no critical points on $S(c)$. Indeed, if we assume by contradiction that there exist $\hat{c}\leq c_2^*$ and a function $\hat{u}\in S(\hat{c})$ being a critical point of $I(u)$ on $S(\hat{c})$. Then by Lemma \ref{lm5.0}, $P(\hat{u})$=0. On the other hand, as \eqref{4.3}, using the Gagliardo-Nirenberg inequality \eqref{p-GN3}, we have
\begin{equation}\label{4.5}
\begin{split}
 0=P(\hat{u})  &=\|\nabla \hat{u}\|_2^{2}+\frac{N(q-2)+2q}{2q}\|\nabla \hat{u}\|_q^{q}-\frac{N(p_q^*-2)+2b}{2p_q^*} \int \frac{|\hat{u}|^p}{|x|^b}dx\\
               &\geq \|\nabla \hat{u}\|_2^{2}+\frac{N(q-2)+2q}{2q}\Big[1-\Big(\frac{\hat{c}}{c_2^*}\Big)^{\frac{p_q^*-q}{2}}\Big] \|\nabla \hat{u}\|_q^{q}\geq \|\nabla \hat{u}\|_2^{2}.
\end{split}
\end{equation} 
This implies that $\|\nabla \hat{u}\|_2=0$, which clearly contradicts with the fact that $\hat{u}\in S(\hat{c})$. Then $I(u)$ has no critical points on $S(c)$ for all $c\leq c_2^*$. The remaining is trivial, then the proof is completed.\\
\end{proof}

\section{THE MASS SUPERCRITICAL CASE}\label{sec-critical}\label{sec-supercritical}
In this section, we consider the mass supercritical case $p_q^*<p<\frac{q(N-b)}{(N-q)^+}$, where $p_q^*=\frac{2(q-b)}{N}+q$. In this case, the functional $I(u)$ is unbounded from below on $S(c)$, see Remark \ref{rek3.1}. Hence one needs to find solutions as local type of minimizers of $I(u)$ on $S(c)$. We shall consider both the existence of ground state solutions and also infinitely many bound state solutions on $S(c)$.

\subsection{Ground state solutions on $S(c)$}
To find a critical point of $I(u)$ on $S(c)$ which has the least energy among all solutions on $S(c)$, inspired by the work of \cite{BJL,J,LH}, we consider the following local type minimization problem:
\begin{equation}\label{localmini}
\gamma(c) := \inf_{u \in V(c)} I(u),\ c>0,
\end{equation}
where $V(c):=\{u\in S(c)\ :\ P(u)=0\}$, with
\begin{equation}\label{Q}
P(u) := \| \nabla u\|_2^2 + \frac{N(q-2)+2q}{2q} \| \nabla u\|_q^q - \frac{N(p-2)+2b}{2p} \int\frac{|u|^p}{|x|^b}dx = 0.
\end{equation}
Note that 
\begin{equation}\label{Qt}
P(u)=\frac{d}{dt} I(u_t) \mid_{t=1} =0 ,\quad u_t= t^{\frac{N}{2}}u(tx) \in S(c).
\end{equation}
In Lemma \ref{lm5.0}, we show that if $u$ is a weak solution of Eq. \eqref{eq1.1}, then necessarily $P(u)=0$. Therefore, ``$P(u)=0$" is referred as a variant of Pohozaev identity. In addition, from Lemma \ref{lm5.1} and Lemma \ref{lm5.2}, we see that for all $c>0$, $V(c)\neq \emptyset$, hence $\gamma(c)$ is well-defined on $\R^+$.

In what follows, we aim to show that for some $c>0$, $\gamma(c)$ is attained by some $u_c\in V(c)$, and $u_c$ is a critical point of $I(u)$ on $S(c)$. Firstly, we prove that
\begin{lemma}\label{lm5.1}
Assume that $p_q^*<p<\frac{q(N-b)}{(N-q)^+}$, then 
\begin{equation*}
\gamma(c)>0,\ \forall\ c>0.
\end{equation*}
\end{lemma}

\begin{proof}
For any $c>0$, let $u\in V(c)$, then from $P(u)=0$ and the $L^p$-Gagliardo-Nirenberg inequality \eqref{p-GN}, we deduce that
\begin{align*}
\frac{N(q-2)+2q}{2q} \| \nabla u\|_q^q &\leq \frac{N(p-2)+2b}{2p} \int\frac{|u|^p}{|x|^b}dx \\
                                    &\leq \frac{N(p-2)+2b}{2p} \mathcal{K}_{N,p,q} \|\nabla u\|_q^{\sigma_{p,q}}c^{\frac{p-\sigma_{p,q}}{2}}.
\end{align*}

Since $\sigma_{p,q}>q$ as $p>p_q^*$, then this implies that 
\begin{equation}\label{5.4}
  \| \nabla u\|_q \geq C(N,p,q,\mathcal{K}_{N,p,q})>0.
\end{equation}
On the other hand, for any $u\in V(c)$, we have
\begin{equation}\label{5.5}
\begin{split}
  I(u)&=\frac{1}{2}\|\nabla u\|_2^{2}+\frac{1}{q}\|\nabla u\|_q^{q}-\frac{1}{p} \int\frac{|u|^{p}}{|x|^{b}}dx\\
      &=\dfrac{Np-2(N+2)+2b}{2[N(p-2)+2b]}\|\nabla u\|_2^{2} +\dfrac{Np-q(N+2)+2b}{q[N(p-2)+2b]}\| \nabla u\|_q^q.
\end{split}
\end{equation}
Since $Np-2(N+2)+2b>Np-q(N+2)+2b>0$ as $p>p_q^*$, then we obtain from \eqref{5.4} and \eqref{5.5} that $\gamma(c)>0$.
\end{proof}

\begin{lemma}\label{lm5.2}
Assume that $p_q^*<p<\frac{q(N-b)}{(N-q)^+}$. Then for any $u \in S(c)$, there exists a unique $t_u >0$, such that $P(u_{t_u})=0$ and $I(u_{t_u})=\max_{t>0} I(u_t)$. Moreover, if $P(u)\leq 0$, then $t_u\in (0,1]$. 

As a consequence, for all $c>0$, we have
\begin{equation}\label{5.5.1}
\gamma(c)=\inf_{u\in S(c)}\max_{t>0}I(u_t).
\end{equation}
\end{lemma}

\begin{proof}
Fixed $u\in S(c)$, we observe that 
\begin{align}\label{5.6}
\begin{split}
P(u_t)&=t^2\Big[\|\nabla u\|_2^{2} +\frac{N(q-2)+2q}{2q}\cdot t^{\frac{N(q-2)+2q}{2}-2}\|\nabla u\|_q^{q} \\
&-\frac{N(p-2)+2b}{2p}\cdot t^{\frac{N(p-2)+2b}{2}-2}\int\frac{|u|^{p}}{|x|^{b}}dx\Big] \triangleq t^2\cdot g(t).
\end{split}
\end{align}
Since $2<\frac{N(q-2)+2q}{2q}<\frac{N(p-2)+2b}{2p}$ as $p>p_q^*$, then it is easy to verify that there exists a unique $t_u>0$ such that $g(t_u)=0$. Equivalently, there exists a unique $t_u>0$, such that $P(u_{t_u})=0$. 

On the other, from \eqref{scaling1}, we have 
\begin{align}\label{5.7}
\begin{split}
\frac{d}{dt}I(u_t)&=t\|\nabla u\|_2^{2} + \frac{N(q-2)+2q}{2q}\cdot\frac{t^{\frac{N(q-2)+2q}{2}-1}}{q} \|\nabla u\|_q^{q} \\
&-\frac{N(p-2)+2b}{2p}\cdot\frac{t^{\frac{N(p-2)+2b}{2}-1}}{p} \int\frac{|u|^{p}}{|x|^{b}}dx=\frac{P(u_t)}{t}.
\end{split}
\end{align}
Therefore, we have 
$$\frac{d}{dt}I(u_t)>0, \mbox{ as } 0<t < t_u,\quad \mbox{ and } \frac{d}{dt}I(u_t)<0, \mbox{ as } t > t_u.$$
Hence, $I(u_{t_u})=\max_{t>0} I(u_t)$.  Moreover, if $P(u)<0$, since $P(u_t)>0$ as $t>0$ small enough, then by continuity we know that $0<t_u<1$. 

To end the proof, we verify \eqref{5.5.1}. Denote 
$$d(c):=\inf_{u\in S(c)}\max_{t>0}I(u_t),$$ 
then by the definition of $\gamma(c)$, for any $\varepsilon>0$, there is $u_{\varepsilon}\in V(c)$, such that $I(u_{\varepsilon})<\gamma(c)+\varepsilon$. By the above proof, $\max_{t>0}I((u_{\varepsilon})_t)=I(u_{\varepsilon})$, this shows that 
$$d(c)\leq \max_{t>0}I((u_{\varepsilon})_t)<\gamma(c)+\varepsilon.$$
Then $d(c)\leq \gamma(c)$. On the other hand, for any $u\in S(c)$, there is $t_u>0$ such that $u_{t_u}\in V(c)$, and 
$$\max_{t>0}I((u_t)=I(u_{t_u})\geq \gamma(c),$$
which implies that $d(c)\geq \gamma(c)$. Thus \eqref{5.5.1} follows.
\end{proof}

\begin{lemma}\label{lm5.3}
Assume that $p_q^*<p<\frac{q(N-b)}{(N-q)^+}$. Then the mapping: $c\mapsto \gamma(c)$ is non-increasing on $\R^+$. Namely, for any $0<c_1<c_2$, $\gamma(c_2)\leq \gamma(c_1)$. 
\end{lemma}

\begin{proof}
It suffices to prove that for any $\varepsilon>0$,
\begin{eqnarray}\label{5.9}
\gamma(c_2)<\gamma(c_1)+\varepsilon.
\end{eqnarray}
Indeed, for $\gamma(c_1)$, there exists $u_1 \in  V(c_1)$, such that 
$$\gamma(c_1)\leq I(u_1)<\gamma(c_1)+\frac{\varepsilon}{2},$$
and by Lemma \ref{lm5.2}, $I(u_1)=\max_{t>0}I((u_1)_t)$.  Since $C_c^{\infty}(\R^N)$ is dense in $X$, then $u_1$ can be assumed to have a bounded compact support. Without loss of generality, we assume that $supp\ u_1\subset B_R$ for some $R>0$. Let $v_0\in X\setminus \{0\}$, such that $supp\ v_0\subset B_{2R}\setminus B_R$, and 
$$v:=\sqrt{\frac{c_2-c_1}{\|v_0\|_2^2}}\cdot v_0.$$
Then $v\in S(c_2-c_1)$. Now for any $\lambda\in (0,1)$, we denote $u_2:=u_1+v_{\lambda}$, with $v_{\lambda}:=\lambda^{\frac{N}{2}}v(\lambda x)$. Then 
$$dist \{supp\ u_1, supp\ v_{\lambda}\}>\frac{R}{\lambda}-R>0.$$ 
Thus,
\begin{align*}
\|u_2\|_2^2&=\|u_1\|_2^2+\|v_{\lambda}\|_2^2=c_2,\\
\|\nabla u_2\|_2^2&=\|\nabla u_1\|_2^2 + \lambda^{2}\|\nabla v\|_2^2,\\
\|\nabla u_2\|_q^q&=\|\nabla u_1\|_q^q + \lambda^{\frac{N(q-2)+2q}{2}}\|\nabla v\|_q^q,\\
\int \frac{|u_2|^p}{|x|^b}dx&=\int \frac{|u_1|^p}{|x|^b}dx + \lambda^{\frac{N(p-2)+2b}{2}}\int \frac{|v|^p}{|x|^b}dx.
\end{align*}
Therefore, as $\lambda>0$ small enough, we have
$$\max_{t>0}I((u_2)_t)\leq \max_{t>0}I((u_1)_t)+\frac{\varepsilon}{2}.$$
This implies from \eqref{5.5.1} that 
$$\gamma(c_2)\leq \max_{t>0}I((u_2)_t)< \gamma(c_1)+\varepsilon.$$
Then \eqref{5.9} is verified.
\end{proof}

\begin{lemma}\label{lm5.5}
Assume that $p_q^*<p<\frac{q(N-b)}{(N-q)^+}$. If the couple $(u_c,\lambda_c)\in S(c)\times \R$ solves weakly the equation:
 \begin{equation}\label{eq5.1}
 -\triangle u -\triangle_q u =\lambda u + \frac{|u|^{p-2}u}{|x|^b}, \quad x\in \mathbb{R}^N.
\end{equation}
Then
\begin{itemize}
  \item [(1)] If $N=1,2$, $p_q^* < p <+\infty$, then $\lambda_c < 0$, for all $c >0$;
  \item [(2)] If $N\geq 3$ with $q<\frac{2(N^2-2b)}{N^2-4}$, $p_q^* < p \leq 2_b^*:=\frac{2(N-b)}{(N-2)}$, then $\lambda_c < 0$, for all $c>0$;
  \item [(3)] As $c\to 0^+$, we have 
  \begin{eqnarray}
\|\nabla u_c\|_2\to +\infty,\ \|\nabla u_c\|_q\to +\infty,\ I(u_c)\to +\infty,\ \lambda_c\to -\infty.
  \end{eqnarray}
\end{itemize}
\end{lemma}

\begin{proof}
For any given $c>0$, by Lemma \ref{lm5.0}, $P(u_c)=0$. In addition, multiplying \eqref{eq5.1} by $u_c$, and integrating by parts, we have
\begin{equation}\label{5.12}
  \lambda_c \|u_c\|_2^2 + \int\frac{|u_c|^p}{|x|^b}dx = \| \nabla u_c\|_2^2 + \| \nabla u_c\|_q^q.
\end{equation}
Thus, using $P(u_c)=0$, we have
\begin{align}\label{5.13}
\begin{split}
  \lambda_c \|u_c\|_2^2  &= \| \nabla u_c\|_2^2 + \| \nabla u_c\|_q^q - \int\frac{|u_c|^p}{|x|^b}dx\\
                    &=\frac{(N-2)p-2(N-b)}{N(p-2)+2b} \| \nabla u_c\|_2^2 + \frac{2[(N-q)p-q(N-b)]}{q[N(p-2)+2b]} \| \nabla u_c\|_q^q\\
                    &=\frac{(N-2)(p-2_b^*)}{N(p-2)+2b} \| \nabla u_c\|_2^2 + \frac{2(N-q)(p-q_b^*)}{q[N(p-2)+2b]} \| \nabla u_c\|_q^q.
\end{split}
\end{align}
Since $2_b^*<q_b^*:=\frac{q(N-b)}{(N-q)^+}$ as $q>2$ and $0<b<\min\{2,N\}$, then if $p \leq 2_b^*$, then clearly $\lambda_c <0$. \\

Again using $P(u_c)=0$ and the $L^q$-Gagliardo-Nirenberg inequality \eqref{p-GN}, we have
\begin{equation*}
 \| \nabla u_c\|_q^q \leq \frac{q[N(p-2)+2b]}{p[N(q-2)+2q]} \int\frac{|u_c|^p}{|x|^b}dx\\
                   \leq C_0(N,p,q,b)\|\nabla u_c \|_q^{\sigma_{p,q}} \| u_c \|_2^{p-\sigma_{p,q}},
\end{equation*}
this leads to
\begin{equation*}
C_0(N,p,q,b)^{-1} \leq  \|\nabla u_c \|_q^{\frac{qN(p-p_q^*)}{N(q-2)+2q}}\| u_c \|_2^{\frac{2(N-q)(q_b^*-p)}{N(q-2)+2q}}.
\end{equation*}
which then shows that 
$$ \|\nabla u_c \|_q \rightarrow +\infty,\ \text{ as }\ c \to 0.$$ 
Further, from $P(u_c)=0$ and \eqref{p-GN2}, we have 
$$\lim_{c\to 0}\int\frac{|u_c|^p}{|x|^b}dx=+\infty,\quad \lim_{c\to 0}\|\nabla u_c \|_2 \rightarrow +\infty.$$
Hence, by \eqref{5.5} and \eqref{5.13}, $I(u_c)\to +\infty,\ \lambda_c\to -\infty$ as $c \to 0$. Then the proof is completed.

\end{proof}

\begin{remark}\label{rek5.1}
When $2_b^* < p < q_b^*$, we conjecture that $\lambda_c < 0$, for all $0<c<c_0$ as some $c_0:=c(N,p,q,b)>0$. In fact, we observe from \eqref{5.13} that 
\begin{equation*}
  \lambda_c \|u_c\|_2^2  = C_1(N,p,q,b) \| \nabla u_c\|_2^2 - C_2(N,p,q,b) \| \nabla u_c\|_q^q,
\end{equation*}
where $C_i(N,p,q,b)>0, i=1,2$. Then the sign of $\lambda_c$ depends on the comparison between $\| \nabla u_c\|_2^2$ and $\| \nabla u_c\|_q^q$. Using $P(u_c)=0$ and the $L^q$-Gagliardo-Nirenberg inequality \eqref{p-GN}, we have
\begin{align*}
\|\nabla u_c\|_2^2 &\leq C_3(N,p,q,b)\| \nabla u_c\|_q^{\sigma_{p,q}}\|u_c\|_2^{p-\sigma_{p,q}}\\
&=C_3(N,p,q,b)\| \nabla u_c\|_q^{\eta}\cdot\Big(\| \nabla u_c\|_q^{\sigma_{p,q}-\eta}\|u_c\|_2^{p-\sigma_{p,q}}\Big),
\end{align*}
where $\eta\in (0,q)$ to be determined. If one could prove that there exist $\eta>0$ and $C_1>0$ independent of $c>0$, such that
\begin{eqnarray}\label{5.14}
\sup_{0<c\leq 1}\| \nabla u_c\|_q^{\sigma_{p,q}-\eta}\|u_c\|_2^{p-\sigma_{p,q}}\leq C_1,
\end{eqnarray}
then 
\begin{equation*}
  \lambda_c \|u_c\|_2^2  \leq \widetilde{C_1}(N,p,q,b) \| \nabla u_c\|_q^{\eta} - C_2(N,p,q,b) \| \nabla u_c\|_q^q,
\end{equation*}
from which we conclude that $\lambda_c < 0$, for all $c>0$ small enough, since $\|\nabla u_c \|_q \rightarrow +\infty$ as $c\to 0$. However, a convinced proof of \eqref{5.14} is still open for us. In the appendix, we provide an exponential decay of $u_c$ but not uniform, which may be helpful to verify \eqref{5.14}.
\end{remark}

\begin{lemma}\label{lm5.8}
Assume that $p_q^*<p<\frac{q(N-b)}{(N-q)^+}$, and $(u_c, \lambda_c) \in X \times \mathbb{R}$ solves
 \begin{equation*}
 -\triangle u -\triangle_q u =\lambda u + \frac{|u|^{p-2}u}{|x|^b}, \quad x\in \mathbb{R}^N,
\end{equation*}
with $I(u_c)=\inf_{u \in V(c)}I(u)=\gamma(c)$. If $\lambda_c <0$, then the mapping: $c\mapsto \gamma(c)$ is strictly decreasing in a right neighborhood of $c$.
\end{lemma}
\begin{proof}
The idea of the proof stems from \cite[Lemma 5.5]{BJL}, which is based on the Implicit Function Theorem. We denote 
$$u_{t, \theta}(x):=\theta^{\frac{N}{2}}t^{\frac{1}{2}}u_c(\theta x)\in S(tc),$$ 
with $\theta>0$ and $t>0$. Define 
\begin{equation}\label{alpha}
\alpha(t, \theta):= I(u_{t, \theta}),
\end{equation}
\begin{equation}\label{alpha}
\beta(t, \theta):= P(u_{t, \theta}).
\end{equation}
By direct calculations, we have
\begin{equation*}
\frac{\partial \alpha(t,\theta)}{\partial t}_{|_{(1,1)}}=\frac{1}{2} \lambda_c c, \quad \frac{\partial \alpha(t,\theta)}{\partial \theta}_{|_{(1,1)}}=0, \quad \frac{\partial^2 \alpha(t,\theta)}{\partial^2 \theta}_{|_{(1,1)}}<0.
\end{equation*}
Then, for any $\delta_t \in \R$, $ \delta_{\theta} \in \R$,
\begin{equation}
\alpha(1+\delta_t, 1+\delta_{\theta})=\alpha(1,1)+\delta_t \frac{\partial \alpha(t,\theta)}{\partial t}_{|_{(\bar t,\bar \theta)}}+\delta_{\theta} \frac{\partial \alpha(t,\theta)}{\partial \theta}_{|_{(\bar t,\bar \theta)}},
\end{equation}
where $|1- \bar t| \leq |\delta _t|$ and $|1- \bar \theta| \leq |\delta_{\theta}|$. If $\lambda_c <0$, then by continuity, for sufficiently small $\delta_{t} >0$ and sufficiently small  $|\delta_{\theta}|,$ we have 
\begin{equation}\label{crucial}
\alpha(1+\delta_t, 1+\delta_{\theta})<\alpha(1,1).
\end{equation}

Note that $\beta(1,1)=0$, if we have 
\begin{equation}\label{implicit}
\frac{\partial \beta(t,\theta)}{\partial \theta}_{|_{(1,1)}}\neq 0.
\end{equation}
then applying the Implicit Function Theorem, for any $\varepsilon>0$, and $t\in [1, 1+\varepsilon]$, there is a continuous function $\theta(t)$, such that $\beta(t,\theta(t))=0$ for all $t \in [1, 1+\varepsilon]$. Thus by \eqref{crucial}, we have
$$\gamma((1+\varepsilon)c)=\inf_{u \in V((1+\varepsilon)c)}I(u)\leq I(u_{1+\varepsilon, \theta(1+\varepsilon)})<I(u_c)=\gamma(c).$$

Hence, to end the proof, we only need to prove \eqref{implicit}. Indeed, by direct calculation,
$$\frac{\partial \beta(t,\theta)}{\partial \theta}_{|_{(1,1)}} = 2\| \nabla u_c\|_2^2 + \frac{[N(q-2)+2q]^2}{4q} \| \nabla u_c\|_q^q - \frac{[N(p-2)+2b]^2}{4p} \int\frac{|u_c|^p}{|x|^b}dx.$$
In view of $P(u_c)=0$, then 
$$\frac{\partial \beta(t,\theta)}{\partial \theta}_{|_{(1,1)}} = \frac{N}{2}\Big(\frac{2(2-b)}{N}+2-p\Big)\| \nabla u_c\|_2^2 + \frac{N[N(q-2)+2q]}{4q}(p_q^*-p)\| \nabla u_c\|_q^q.$$
Then, since $p> p_q^*:=\frac{2(q-b)}{N}+q$, $q>2$, then the coefficients in the above equality are both strictly negative, and also $\| \nabla u_c\|_2^2, \| \nabla u_c\|_q^q\neq 0$, then \eqref{implicit} is verified.
\end{proof}

\begin{proposition}\label{prop5.1}
Assume that $p_q^*<p<\frac{q(N-b)}{(N-q)^+}$. Then $\gamma(c)$ admits a minimizer, provided that $p$ and $c$ satisfy one of the following conditions:
\begin{itemize}
  \item [(1)] If $N=1,2$, $p_q^* < p <+\infty$, and $c >0$;
  \item [(2)] If $N\geq 3$ with $q<\frac{2(N^2-2b)}{N^2-4}$, $p_q^* < p < 2_b^*:=\frac{2(N-b)}{(N-2)}$, and $c>0$.
\end{itemize}
Furthermore, any minimizer of $\gamma(c)$ can be assumed to be non-negative, radially symmetric and radially decreasing with respect to some point.
\end{proposition}

\begin{proof}
Let $\{u_n\}\subset V(c)$ be a minimizing sequence of $\gamma(c)$. Then, we observes from \eqref{5.5} that $\{u_n\}$ is bounded in $X$, thus $u_n \rightharpoonup u_0$ in $X$, and due to Lemma \ref{lm_ap1}, $\int \frac{|u_n|^p}{|x|^b} \rightarrow \int \frac{|u_0|^p}{|x|^b}$. We claim that $u_0 \neq 0$. Indeed, if $u_0=0$, then from $P(u_n)=0$, we deduce that $I(u_n)=o_n(1)=\gamma(c)+o_n(1)$, which contradicts with Lemma \ref{lm5.1}. 

Now, we show that $u_n \rightarrow u_0$ in $X$.
In fact, using the Brezis-Lieb Lemma \cite{BL} and Lemma \ref{lm_ap1}, we have the following decomposition:
\begin{equation*}
\| u_n-u_0 \|_2^2 + \| u_0 \|_2^2 = \| u_n \|_2^2 + o_n(1),
\end{equation*}
\begin{equation}
  I( u_n-u_0 ) + I( u_0 ) = I( u_n ) + o_n(1),   \tag{$ a $}
\end{equation}
\begin{equation}
  P( u_n-u_0 ) + P( u_0 ) = P( u_n ) + o_n(1).   \tag{$ b $}
\end{equation}
In addition,
\begin{equation}
\begin{split}
  &I( u_n-u_0 ) - \frac{2}{N(p-2)+2b} P( u_n-u_0 )\\
  =& \frac{Np-2(N+2)+2b}{2[N(p-2)+2b]}\|\nabla u_n-u_0\|_2^{2} +\frac{Np-q(N+2)+2b}{q[N(p-2)+2b]}\| \nabla u_n-u_0\|_q^q.
\end{split}
\tag{$ c $}
\end{equation}
Since $u_n \rightharpoonup u_0$ in $X$, then $u_0 $ is a weak solution of Eq. \eqref{eq1.1} and then by Lemma \ref{lm5.0}, $P(u_0) = 0$. Thus $I(u_0) \geq \gamma ( \| u_0 \|_2^2) \geq \gamma(c)$. Therefore, from $(a), (b)$ and $(c)$, we have
\begin{equation*}
\frac{Np-2(N+2)+2b}{2[N(p-2)+2b]}\|\nabla u_n-u_0\|_2^{2} +\frac{Np-q(N+2)+2b}{q[N(p-2)+2b]}\| \nabla u_n-u_0\|_q^q = o_n(1).
\end{equation*}
Note that  $Np-2(N+2)+2b>Np-q(N+2)+2b>0$ as $p>p_q^*$. Hence, $\| \nabla (u_n-u_0) \|_2^2  = o_n(1) $ and $\| \nabla (u_n-u_0) \|_q^q  = o_n(1) $. Thus,
\begin{equation*}
  \| \nabla (u_n-u_0) \|_2^2 + \| \nabla u_0 \|_2^2 = \| \nabla u_n \|_2^2 + o_n(1),
\end{equation*}
\begin{equation*}
  \| \nabla (u_n-u_0) \|_q^q + \| \nabla u_0 \|_q^q = \| \nabla u_n \|_q^q + o_n(1),
\end{equation*}
yield that $ \| \nabla u_n \|_2 \rightarrow \| \nabla u_0 \|_2 $, $ \| \nabla u_n \|_q \rightarrow \| \nabla u_0 \|_q $. Hence, from Lemma \ref{lm5.3}, we conclude that
\begin{equation}\label{5.29}
P(u_0)=0,\ I(u_0)=\lim_{n\to \infty}I(u_n)=\gamma(\|u_0\|_2^2)=\gamma(c).
\end{equation}
In particular, $u_0$ is a minimizer of $\gamma(\|u_0\|_2^2)$, with $0<\|u_0\|_2^2\leq c$. Thus, from Lemma \ref{lm5.5} Lemma \ref{lm5.8} we deduce that $\|u_0\|_2^2=c$. Hence, $u_n \rightarrow u_0$ in $X$, and $u_0$ is a minimizer of $\gamma(c)$.

In addition, for given $c>0$, let $u_c$ is a minimizer of $\gamma(c)$. Clearly, $|u_c|$ is also a minimizer, hence without loss of generality, we can assume that $u_c\geq 0$.

Let $u_c^*$ be the Schwartz symmetrization of $u_c$, then by the Polya-Szeg\"o inequality (see e.g. \cite{LL}) and Lemma \ref{lm_ap2}, we have that
\begin{align*}
\|u_c^*\|_2^2=\|u_c\|_2^2=c,\ \|\nabla u_c^*\|_2^2 \leq \|\nabla  u_c\|_2^2,\ \|\nabla u_c^*\|_q^q\leq \|\nabla u_c\|_q^q,\ \int \frac{|u_c|^p}{|x|^b}dx \leq \int \frac{|u_c^*|^p}{|x|^b}dx.
\end{align*}
This implies that
\begin{align*}
&I(u_c^*)\leq I(u_c)=\gamma(c),\\
&P(u_c^*)\leq P(u_c)=0.
\end{align*}

We claim that $P(u_c^*)=0$. Indeed, if we suppose that $P(u_c^*)<0$, then by Lemma \ref{lm5.2}, there exists a unique $t_c\in (0,1)$, such that $P((u_c^*)_{t_c})=0$. Thus, from \eqref{5.5}
\begin{align*}
\gamma(c)&\leq I((u_c^*)_{t_c})\\
&=\dfrac{Np-2(N+2)+2b}{2[N(p-2)+2b]}t_c^{2}\|\nabla u_c^*\|_2^{2} +\dfrac{Np-q(N+2)+2b}{q[N(p-2)+2b]}t_c^{\frac{N(q-2)+2q}{2}}\| \nabla u_c^*\|_q^q\\
&<\dfrac{Np-2(N+2)+2b}{2[N(p-2)+2b]}\|\nabla u_c\|_2^{2} +\dfrac{Np-q(N+2)+2b}{q[N(p-2)+2b]}\| \nabla u_c\|_q^q\\
&=I(u_c)=\gamma(c),
\end{align*}
which is a contradiction. Hence, $P(u_c^*)=0$, then by $I(u_c^*)\leq I(u_c)=\gamma(c)$, $u_c^*$ is also a minimizer of $\gamma(c)$. Moreover,
\begin{eqnarray}\label{5.19}
\int\frac{|u_c|^p}{|x|^b}dx =\int \frac{|u_c^*|^p}{|x|^b}dx.
\end{eqnarray}
Thus, using Lemma \ref{lm_ap2}, we deduce from \eqref{5.19} that $u_c=u_c^*$. Hence $u_c$ is radially symmetric and decreasing with respect to some point. 
\end{proof}

In the following lemma, we prove the behavior of $\gamma(c)$ as $c\to \infty$.
\begin{lemma}\label{lm5.55}
Assume that $p$ satisfies one of the following conditions:
\begin{itemize}
  \item [(1)] If $N=1,2$, $p_q^* < p <+\infty$;
  \item [(2)] If $N\geq 3$ with $q<\frac{2(N^2-2b)}{N^2-4}$, $p_q^* < p < 2_b^*:=\frac{2(N-b)}{(N-2)}$.
\end{itemize}
Then $\lim\limits_{c\to \infty}\gamma(c)=0$. In particular, if $u_c\in S(c)$ is a minimizer of $\gamma(c)$, then 
\begin{align}\label{5.190}
I(u_c)=\gamma(c)\to 0,\ \|\nabla u_c\|_2\to 0,\ \|\nabla u_c\|_q\to 0, \quad \mbox{ as } c\to \infty.
\end{align}
\end{lemma}

\begin{proof}
Let $u_0\in S(1)\cap L^{\infty}(\R^N)$, and $u_0^{t_0}:=t_0u(t_0^{-1}x)$ with $t_0=c^{\frac{1}{N+2}}$, thus $u_0^{t_0}\in S(c)$. From Lemma \ref{lm5.2}, there exists a unique $t_c>0$, such that $P((u_0^{t_0})_{t_c})=0$, then using \eqref{scaling1} and \eqref{3.11} we have
\begin{align}\label{5.191}
0<\gamma(c)\leq I((u_0^{t_0})_{t_c})\leq \frac{c^{\frac{N}{N+2}}t_c^{2}}{2}\|\nabla u_0\|_2^{2}+\frac{c^{\frac{N}{N+2}}t_c^{\frac{N(q-2)+2q}{2}}}{q}\|\nabla u_0\|_q^{q}.
\end{align}
Clearly, if there holds that 
\begin{align}\label{5.192}
\lim\limits_{c\to \infty}c^{\frac{N}{N+2}}t_c^{2}=0. 
\end{align}
then $\lim\limits_{c\to \infty}\gamma(c)=0$ follows immediately from \eqref{5.191}. Now, we verify \eqref{5.192} as $p$ satisfies the conditions given in this lemma. Indeed, using  $P((u_0^{t_0})_{t_c})=0$, we have
\begin{align*}
c^{\frac{N}{N+2}}t_c^{2}\|\nabla u_0\|_2^{2} + c^{\frac{N}{N+2}}t_c^{\frac{N(q-2)+2q}{2}}\frac{N(q-2)+2q}{2q}\|\nabla u_0\|_q^{q}\\
=c^{\frac{p-b+N}{N+2}}t_c^{\frac{N(p-2)+2b}{2}}\frac{N(p-2)+2b}{2p}\int\frac{|u_0|^{p}}{|x|^{b}}dx,
\end{align*}
which implies the following two equalities:
\begin{align}\label{5.193}
\|\nabla u_0\|_2^{2} + t_c^{\frac{N(q-2)+2q-4}{2}}\frac{N(q-2)+2q}{2q}\|\nabla u_0\|_q^{q}=c^{\frac{p-b}{N+2}}t_c^{\frac{N(p-2)+2b-4}{2}}\frac{N(p-2)+2b}{2p}\int\frac{|u_0|^{p}}{|x|^{b}}dx.
\end{align}
\begin{align}\label{5.194}
t_c^{-\frac{N(p-2)+2b-4}{2}}\|\nabla u_0\|_2^{2} + t_c^{-\frac{N(p-p_q^*)}{2}}\frac{N(q-2)+2q}{2q}\|\nabla u_0\|_q^{q}=c^{\frac{p-b}{N+2}}\frac{N(p-2)+2b}{2p}\int\frac{|u_0|^{p}}{|x|^{b}}dx.
\end{align}

From \eqref{5.194}, we observe that when $p>p_q^*$, then $t_c\to 0$ as $c\to \infty$. Thus we deduce from \eqref{5.193} that 
$$\lim\limits_{c\to \infty}c^{\frac{p-b}{N+2}}t_c^{\frac{N(p-2)+2b-4}{2}}=C_1>0.$$
or equivalently 
\begin{eqnarray}\label{5.195}
\lim\limits_{c\to \infty}\Big(c^{\frac{N}{N+2}}t_c^2\Big)^{\frac{p-b}{N}}\cdot t_c^{\beta}=C_1>0,
\end{eqnarray}
with $\beta:=\frac{N+2}{2N}[(N-2)p-2(N-b)]$, $C_1>0$ independent of $c$. Since we have proved that $t_c\to 0$ as $c\to \infty$, hence, if $\beta<0$, then  \eqref{5.195} implies \eqref{5.192}. By the definition of $\beta$, it is trivial that $\beta<0$ provides that $p$ satisfies the assumption of the lemma. 

Having proved that $\lim\limits_{c\to \infty}\gamma(c)=0$, then \eqref{5.190} follows immediately from \eqref{5.5}. Thus, the proof is completed.
\end{proof}

\begin{remark}
As for the case that $N\geq 3$ and $p=2_b^*$, we observe from \eqref{5.195} that $\lim\limits_{c\to \infty} c^{\frac{N}{N+2}}t_c^2=C_2$, with $C_2>0$ independent of $c>0$. Thus using  $P((u_0^{t_0})_{t_c})=0$ we can calculate that $\lim\limits_{c\to \infty}I((u_0^{t_0})_{t_c})=\widetilde{C_2}>0$, but nothing can be concluded from this and \eqref{5.191}. 
\end{remark}


Fanally, we prove that a minimizer of $\gamma(c)$ is indeed a critical point of $I(u)$ on $S(c)$. Namely the constraint ``$P(u)=0$" is referred as a natural constraint. 
\begin{lemma}\label{lm5.6}
Assume that $p_q^*<p<\frac{q(N-b)}{(N-q)^+}$. Then any critical point of $I|_{V(c)}$ is exactly a critical point of $I|_{S(c)}$.
\end{lemma}

\begin{proof}
Suppose that $u$ is a critical point of $I \mid _{V(c)}$, then by the Lagrange multiplier theory, there exist $\lambda\in \R$ and $\mu\in \R$, such that
\begin{equation}\label{5.20}
 I'(u) - \lambda u -\mu P'(u) = 0 \quad in \  X^*.
\end{equation}
Then we only need to prove that $\mu=0$. Note that \eqref{5.20} is equivalent to
\begin{equation}\label{5.21}
 (2\mu -1)\triangle u + \Big[\mu\frac{N(q-2)+2q}{2}-1\Big] \triangle_q u + \Big[\mu\frac{N(p-2)+2b}{2}-1\Big] \frac{|u|^{p-2}u}{|x|^b} - \lambda u =0.
\end{equation}
Multiplying \eqref{5.21} by $u$ and $x\cdot \nabla u$ respectively, and then integrating by part, we have
\begin{equation}\label{5.22}
\begin{split}
(2\mu -1)\|\nabla u \|_2^2 &+\Big [\mu\frac{N(q-2)+2q}{2}-1\Big]\|\nabla u\|_q^q \\
 &-\Big[\mu\frac{N(p-2)+2b}{2}-1\Big]\int \frac{|u|^{p}}{|x|^b}dx + \lambda \|u \|_2^2 = 0,
\end{split}
\end{equation}

\begin{equation}\label{5.23}
\begin{split}
 \frac{(2\mu -1)(N-2)}{2}\|\nabla u \|_2^2 +\Big [\mu\frac{N(q-2)+2q}{2}-1\Big]\frac{N-q}{q}\|\nabla u\|_q^q \\
 +\Big[\mu\frac{N(p-2)+2b}{2}-1\Big] \frac{b-N}{p} \int \frac{|u|^{p}}{|x|^b}dx + \frac{N\lambda}{2} \|u \|_2^2 = 0.
\end{split}
\end{equation}
From \eqref{5.22} and \eqref{5.23}, we have
\begin{equation}\label{5.24}
\begin{split}
\mu\Big[4\|\nabla u \|_2^2 &+\frac{N(q-2)+2q}{2}\cdot \frac{N(q-2)+2q}{q}\|\nabla u\|_q^q \\
 &-\frac{N(p-2)+2b}{2}\cdot\frac{N(p-2)+2b}{p}\int \frac{|u|^{p}}{|x|^b}dx\Big]-2P(u)=0.
\end{split}
\end{equation}
Since $P(u)=0$, then \eqref{5.24} is reduced as 
\begin{equation}\label{5.25}
\mu\Big[(p-\frac{4-2b}{N}-2)\|\nabla u \|_2^2 +\frac{N(q-2)+2q}{2}\cdot (p-p_q^*)\|\nabla u\|_q^q\Big]=0.
\end{equation}
Since $p_q^*=\frac{2(q-b)}{N}+q>\frac{2(2-b)}{N}+2$ as $q>2$, then when $p>p_q^*$, we conclude from \eqref{5.25} that $\mu = 0$.
\end{proof}

\subsection{Infinitely many bound state solutions on $S(c)$}

As a byproduct of the above lemmas in the mass-supercritical case, inspired by the work of Bartsch et.al \cite{BV}, we shall use the fountain theorem to prove the existence of infinitely many bound state normalized solutions of \eqref{eq1.1}. Similar results also have been established for the Schr\"odinger-Poisson-Slater equations in \cite{Luo}, for a class of nonlinear Choquard equations in \cite{BLL}, and recently for a type of $(2,q)$-Laplacian Equation with homogenous nonlinearity in \cite{BY}.

Let $\{V_n\}$ be a strictly increasing sequence of finite-dimensional linear subspaces in $X$, such that $\bigcup_n V_n$ is dense in $X$. We denote by $V_{n}^{\perp}$ the orthogonal space of $V_n$ in $H_r^1(\R^3)$. Then

\begin{lemma}\label{lm5.7}
Assume that $p_q^*<p<\frac{q(N-b)}{(N-q)^+}$. Then there holds that
\begin{equation*}
   \mu_n:= \inf_{u \in V_{n-1}^\bot} \frac{\|\nabla u\|_2^2+\|\nabla u\|_q^q+\|u\|_2^2}{(\int\frac{|u|^p}{|x|^b}dx)^{\frac{2}{p}}} = \inf_{u \in V_{n-1}^\bot} \frac{\|u\|_X^2}{ (\int\frac{|u|^p}{|x|^b}dx)^{\frac{2}{p}} }\rightarrow +\infty, \ \mbox{as}\ n \rightarrow +\infty.
\end{equation*}
\end{lemma}

\begin{proof}
We apply the same idea from \cite[Lemma 2.1]{BV} to prove this lemma. Indeed, if we argue by contradiction to suppose that there exists $\{u_n\}\subset X$, such that $u_n\in V_{n-1}^{\perp}$, $\int \frac{|u_n|^p}{|x|^b}dx=1$, $\|u_n\|_X=l<+\infty$. Then, up to a subsequence, there exists $u_0\in X$, such that $u_n\rightharpoonup u_0$ in $X$. By Lemma \ref{lm_ap1}, $\int \frac{|u_n|^p}{|x|^b}dx \to \int \frac{|u_0|^p}{|x|^b}dx=1$. Since $\bigcup_n V_n$ is dense in $X$ and $V_n$ is strictly increasing, then let $v\in X$ and $\{v_n\}\subset V_{n-1}^{\perp}$, such that $v_n \to v$ in $X$. Thus,
$$|\langle u_n, v \rangle_X|\leq |\langle u_n, v-v_n \rangle_X|+|\langle u_n, v_n \rangle_X|\leq \|u_n\|_X\cdot \|v-v_n\|_X\to 0.$$
Hence, $u_n\rightharpoonup 0=u_0$ in $X$, which is impossible, since we have $\int \frac{|u_0|^p}{|x|^b}dx=1$.
\end{proof}

Now, for any $n\in \mathbb{N}$, we define
$$\rho_n:=L^{-\frac{2}{p-2}}\cdot \mu_n^{\frac{p}{p-2}},\ \mbox{with}\ L:=\max_{x,y>0}\dfrac{(x+y+1)^p}{x^p+y^p+1},$$
and 
$$B_n:=\{u\in S(c)\cap V_{n-1}^{\perp} : \ \|\nabla u\|_2^2+\|\nabla u\|_q^q=\rho_n\}.$$

Denote 
$$b_n:=\inf_{u\in B_n}I(u).$$
Then, we have
\begin{lemma}\label{lm5.9}
Assume that $p_q^*<p<\frac{q(N-b)}{(N-q)^+}$. Then $b_n\to +\infty$ as $n\to \infty$. Without loss of generality, we can assume that $b_n \geq 1$ for all $n \in \N$.
\end{lemma}

\begin{proof}
For any $u\in B_n$, we have that
\begin{align*}
I(u) &= \frac{1}{2}\|\nabla u\|_2^{2}+\frac{1}{q}\|\nabla u\|_q^{q}-\frac{1}{p} \int\frac{|u|^{p}}{|x|^{b}}dx\\
     &\geq \frac{1}{2}\|\nabla u\|_2^{2}+\frac{1}{q}\|\nabla u\|_q^{q}-\frac{1}{p\mu_n^{\frac{p}{2}}} \Big(\|\nabla u\|_2^{2}+\|\nabla u\|_q^{q}+c\Big)^{\frac{p}{2}}\\
     &\geq \frac{1}{2}\|\nabla u\|_2^{2}+\frac{1}{q}\|\nabla u\|_q^{q}-\frac{L}{p\mu_n^{\frac{p}{2}}} \Big(\|\nabla u\|_2^{p}+\|\nabla u\|_q^{\frac{pq}{2}}+c^{\frac{p}{2}}\Big)\\
     &\geq (\frac{1}{q}-\frac{1}{p})\rho_n-\dfrac{Lc^{\frac{p}{2}}}{p\mu_n^{\frac{p}{2}}}.
\end{align*}
Then $b_n\geq (\frac{1}{q}-\frac{1}{p})\cdot L^{-\frac{2}{p-2}}\cdot \mu_n^{\frac{p}{p-2}}-\dfrac{Lc^{\frac{p}{2}}}{p\mu_n^{\frac{p}{2}}}$, which verifies from Lemma \ref{lm5.7} that $b_n\to +\infty$ as $n\to \infty$. Thus, for some $n_0>0$ large enough, $b_n \geq 1$ for any $n \geq n_0$. We consider the sequence beginning from $n_0$, then the proof is finished.
\end{proof}

To set up the min-max scheme of the functional on $V_n$, we introduce the map
\begin{eqnarray*}\label{kappa}
\kappa: X\times\R &\longrightarrow& X \nonumber \\
(u,\ \theta) &\longmapsto& \kappa(u,\theta)(x):= e^{\frac{N}{2}\theta}u(e^{\theta} x).
\end{eqnarray*}
Then, for any given $u\in S(c)$, we have $\kappa(u,\theta)\in S(c)$ for all $\theta\in \R$. Denote 
$$T(u):=\|\nabla u\|_2^2+\|\nabla u\|_q^q,$$ 
and let $t=e^{\theta}$ in \eqref{scaling1}, then we observe from \eqref{scaling1} that, when $p_q^*<p<\frac{q(N-b)}{(N-q)^+}$, then 
\begin{equation}\label{behaviscaling}
\left\{
\begin{array}{l}
T(\kappa(u,\theta))\to 0,\quad I(\kappa(u,\theta))\to 0,\quad \ \mbox{ as } \theta\to -\infty, \\
T(\kappa(u,\theta))\to +\infty,\ I(\kappa(u,\theta))\to -\infty,\ \mbox{ as } \theta\to +\infty.
\end{array}
\right.
\end{equation}
Thus, for each $n \in \N$, there exists $\theta_n>0$ large enough, such that
\begin{eqnarray}\label{g_n}
\bar{g}_n:\ [0,1]\times (S(c)\cap V_n) \to S(c), \ \  \bar{g}_n(t, u) = \kappa (u, (2t-1)\theta_n)
\end{eqnarray}
satisfies
\begin{equation}\label{pro-g_n}
\left\{
\begin{array}{l}
T(\bar{g}_n(0,u))< \rho_n<T(\bar{g}_n(1,u)),\\
\max\limits_{u\in S(c)\cap V_n}\{I(\bar{g}_n(0,u)), I(\bar{g}_n(1,u)) \}< 1\leq b_n.
\end{array}
\right.
\end{equation}
\medskip

Now we define
\begin{align*}\label{declass}
\begin{split}
\Gamma_n: =\Big\{g:\ &[0,1]\times (S(c)\cap V_n) \to S(c)\ |\ g \mbox{ is continuous, odd in } u \\
&\mbox{and such that } \forall u:\ g(0,u)= \bar{g}_n(0,u),\ g(1,u)=\bar{g}_n(1,u)\Big\}.
\end{split}
\end{align*}
Clearly $\bar{g}_n \in \Gamma_n$. Now we give the following intersection lemma.
\begin{lemma}\label{lm5.10}
Assume that $p_q^*<p<\frac{q(N-b)}{(N-q)^+}$. Then for each $n\in \N$,
\begin{eqnarray}\label{gamma_n}
\gamma_n(c) := \inf_{g \in \Gamma_n}\max\limits_{\substack{0\leq t\leq 1 \\ u\in S(c)\cap V_n}}I(g(t,u))\geq b_n.
\end{eqnarray}
In particular, 
\begin{equation}\label{gamma_n1}
\gamma_n(c) \geq b_n >  \max \Big\{ \max_{u\in S(c)\cap V_n} I(g(0,u)),  \max_{u\in S(c)\cap V_n} I(g(1,u))\Big\}.
\end{equation}
\end{lemma}
\begin{proof}
To prove \eqref{gamma_n}, the key is to verify that for each $g \in \Gamma_n$ there exists a pair $(t,u)\in [0,1]\times (S(c)\cap V_n)$, such that $g(t,u)\in B_n$. Based on \eqref{pro-g_n}, this can be proved exactly the same as in \cite[Lemma 5.9]{BY}, with only some modifications in the notation. Here we omit the details. In addition, due to \eqref{pro-g_n}, then \eqref{gamma_n1} is trivial.
\end{proof}

In what follows, we prove that for each $n\in \N$, $\gamma_n(c)$ is a critical value of $I$ on $S(c)$. To this aim, we first prove that there exists a bounded Palais-Smale sequence at the level $\gamma_n(c)$. From now on, let us fix an arbitrary $n\in \N$.
\begin{lemma}\label{lm-PSC}
Assume that $p_q^*<p<\frac{q(N-b)}{(N-q)^+}$. For any given $c>0$, there exists a sequence $\{u_k\}\subset S(c)$ satisfying
\begin{eqnarray}\label{PSC}
\left\{
\begin{array}{l}
I(u_k)\to \gamma_n(c), \\
I'|_{S(c)}(u_k)\to 0, \quad \mbox{ as } k\to \infty,\\
P(u_k)\to 0,
\end{array}
\right.
\end{eqnarray}
where $P(u)$ is given in \eqref{Q}. In particular, $\{u_k\}$ is bounded in $X$.
\end{lemma}
\begin{proof}
 Indeed, due to Lemma \ref{lm5.9} and Lemma \ref{lm5.10}, by letting $E=S(c)\cap V_n$, and 
 $$\gamma(0)=\max_{u\in S(c)\cap V_n} I(g(0,u)),\ \gamma(1)=\max_{u\in S(c)\cap V_n} I(g(1,u)),$$ 
then this lemma is a consequence of \cite[Theorem 4.1]{CT}. In addition, since for all $k>0$, $P(u_k)=0$, then from \eqref{5.5}, we see that $\{u_k\}$ is bounded in $X$.  
\end{proof}

\begin{proposition}\label{prop-compact}
Assume that $p_q^*<p<\frac{q(N-b)}{(N-q)^+}$. Let $\{u_k\}\subset S(c)$ be the Palais-Smale sequence obtained in Lemma \ref{lm-PSC}. Then there exist $\lambda_n\in \R$ and $u_n\in X$, such that, up to a subsequence, as $k\to \infty$,
\begin{enumerate}
  \item [(1)] $u_k \rightharpoonup u_n \neq 0$, in $X$;
  \item [(2)] $ -\triangle u_k -\triangle_q u_k -\lambda_n u_k - \frac{|u_k|^{p-2}u_k}{|x|^b} \to 0$,  in $X^*$;
  \item [(3)] $-\triangle u_n -\triangle_q u_n -\lambda_n u_n - \frac{|u_n|^{p-2}u_n}{|x|^b} = 0$,  in $X^*$.
\end{enumerate}
Moreover, if $\lambda_n <0$, then we have 
$$u_k \to u_n,\ \mbox{ in } X.$$ 
In particular, $||u_n||_2^2 = c$, $I(u_n)=\gamma_n(c)$ and $I'(u_n)-\lambda_n u_n = 0$ in $X^*$.
\end{proposition}

\begin{proof}
By Lemma \ref{lm-PSC}, $\{u_k\}\subset S(c)$ is bounded in $X$, then up to a subsequence, there exists $u_n\in X$, such that $u_k \underset{k}{\rightharpoonup} u_n,\ \mbox{ in } X$, and by Lemma \ref{lm_ap1}, 
\begin{align}\label{5.32}
\int \frac{|u_k|^p}{|x|^b}dx\to \int \frac{|u_n|^p}{|x|^b}dx.
\end{align}

We claim that $u_n \neq 0$. Indeed if $u_n =0$, then $\int \frac{|u_k|^p}{|x|^b}dx\to 0$. Using $P(u_k)\to 0$, we have $\|\nabla u_k\|_2^2\to 0$ and $\|\nabla u_k\|_q^q \to 0$. Thus $I(u_k)\to 0$ and this contradicts with the fact that $\lim_{k\to \infty}I(u_k)=\gamma_n(c) \geq b_n \geq1$. Thus the claim is verified.\medskip

Following \cite[Lemma 3]{BELI}, (see also \cite[Proposition 4.1]{BJL}), we have
\begin{eqnarray*}
I'|_{S(c)}(u_k) &\longrightarrow & 0 \ \mbox{ in } X \\
&\Longleftrightarrow & I'(u_k)-\langle I'(u_k),u_k \rangle u_k \to 0 \mbox{ in } X.\\
\end{eqnarray*}
Thus, for any $w\in X$,
\begin{align}\label{5.33}
\begin{split}
\langle I'(u_k)-\langle I'(u_k),u_k \rangle u_k, w  \rangle &=\int  \nabla u_k  \nabla w dx + \int  |\nabla u_k|^{q-2}\nabla u_k \nabla w dx\\
&-\int \frac{|u_k|^{p-2}u_kw}{|x|^b}dx -\lambda_k \int u_k(x)w(x)dx \to 0,
\end{split}
\end{align}
with
\begin{equation}\label{lambda}
\lambda_k=\frac{1}{\|u_k\|_2^2}\Big\{\|\nabla u_k\|_2^2+\|\nabla u_k\|_q^q- \int \frac{|u_k|^p}{|x|^b}dx\Big\}.
\end{equation}

Since $\{u_k\}$ is bounded in $X$, then clearly, $\{\lambda_k\}$ is bounded in $\R$. Thus up to a subsequence, $\lambda_k\to \lambda_n$ for some $\lambda_n\in \R$. Therefore, the point $(2)$ follows from \eqref{5.33}. Moreover, since $u_k\rightharpoonup u_n$ in $X$, then the point $(2)$ implies immediately the point $(3)$.  

Since \eqref{5.32}, then from $(2)$ and $(3)$, we have 
$$\|\nabla u_k\|_2^2+\|\nabla u_k\|_q^q-\lambda_n \|u_k\|_2^2  \underset{k}{\to} \|\nabla u_n\|_2^2+\|\nabla u_n\|_q^q-\lambda_n \|u_n\|_2^2 .$$
If $\lambda_n<0$, then from $u_k \underset{k}{\rightharpoonup} u_n$ in $X$, we have
$$\|\nabla u_k\|_2^2 \to \|\nabla u_n\|_2^2,\  \|\nabla u_k\|_q^q\to \|\nabla u_n\|_q^q,\ \|u_k\|_2^2\to \|u_n\|_2^2=c.$$
Thus $u_k\underset{k}{\to}u_n$ in $X$, and in particular, $||u_n||_2^2 =c$, $I(u_n)=\gamma_n(c)$ and $I'(u_n)-\lambda_n u_n = 0$ in $X^*$.\\
\end{proof}


\section{Proofs of the main results}\label{sec-proofs}
\begin{proof}[Proof of Theorem \ref{I-th-GN}]
The proof is exactly the same as the one of Lemma \ref{th-main} in Section \ref{sec-GN}. 
\end{proof}

\begin{proof}[Proof of Theorem \ref{I-th-subcritical}]
Theorem \ref{I-th-subcritical} is a consequence of Lemma \ref{lm3.1}, Proposition \ref{prop2.1}, Corollary \ref{col3.1}, and Lemma \ref{lm3.2}. 
\end{proof}

\begin{proof}[Proof of Theorem \ref{I-th-critical}]
Theorem \ref{I-th-critical} has been proved in the proof of Lemma \ref{lm4.1}.
\end{proof}

\begin{proof}[Proof of Theorem \ref{I-th-supercritical}]
Theorem \ref{I-th-supercritical} is a consequence of Proposition \ref{prop5.1}, Lemma \ref{lm5.55} and Lemma \ref{lm5.6}. 
\end{proof}

\begin{proof}[Proof of Theorem \ref{I-th-multiplicity}]
Theorem \ref{I-th-multiplicity} is a consequence of Lemma \ref{lm5.9}, Lemma \ref{lm5.10}, Lemma \ref{lm-PSC}, Lemma \ref{lm5.5} and Proposition \ref{prop-compact}. 
\end{proof}

\section{appendix}\label{sec-apt}
In this section, we give the proof of some technical lemmas, which are used in our main proofs. 
\begin{lemma}\label{lm_ap1}
Assume that $2<p<\frac{q(N-b)}{(N-q)^+}$, with $0 < b < \min\{2,N\}$, $N\geq1$, $q\geq 2 $, and that $\{u_n\}\subset X$ and $u\in X$, being such that
$$u_n\rightharpoonup u,\ \mbox{ in } X^*,$$
then
\begin{eqnarray}\label{ap1.1}
\lim_{n\to \infty}\int \frac{|u_n|^p}{|x|^b}dx= \int \frac{|u|^p}{|x|^b}dx.
\end{eqnarray}
Moreover, we have the following Brezis-Lieb decomposition:
\begin{align}\label{ap1.2}
\int \frac{|u_n|^p}{|x|^b}dx-\int \frac{|u_n-u|^p}{|x|^b}dx-\int \frac{|u|^p}{|x|^b}dx=o_n(1),\quad \mbox{ as } n\to \infty,
\end{align}
\begin{align}\label{ap1.3}
\|\nabla u_n\|_q^q-\|\nabla (u_n-u)\|_q^q-\|\nabla u\|_q^q=o_n(1),\quad \mbox{ as } n\to \infty.
\end{align}
\end{lemma}

\begin{proof}
First, we claim that
\begin{eqnarray}\label{ap1.4}
\int |x|^{-b}|u_n-u|^pdx=o_n(1).
\end{eqnarray}
Indeed, since $u_n\rightharpoonup u\ \mbox{ in } X$, then $\{u_n\}$ is bounded in $X$, and $u_n-u\to 0$ in $L^r_{loc}(\R^N)$, for all $r\in [2, \frac{q(N-b)}{(N-q)^+})$, thus for given $R>0$ to be determined, we have
$$\int |x|^{-b}|u_n-u|^{p}dx=\int_{|x|\leq R} |x|^{-b}|u_n-u|^{p}dx+\int_{|x|\geq R} |x|^{-b}|u_n-u|^{p}dx.$$
Using the H\"older inequality, we have for $\varepsilon>0$ small enough, that
\begin{align*}
\int_{|x|\leq R} |x|^{-b}|u_n-u|^{p}dx&\leq \Big(\int_{|x|\leq R} |x|^{-\frac{bN}{b+\varepsilon}}dx\Big)^{\frac{b+\varepsilon}{N}} \Big(\int_{|x|\leq R} |u_n-u|^{\frac{Np}{N-b-\varepsilon}}dx \Big)^{\frac{N-b-\varepsilon}{N}},\\
\int_{|x|\geq R} |x|^{-b}|u_n-u|^{p}dx &\leq \frac{1}{R^b}\int |u_n-u|^{p}dx.
\end{align*}
For given $p<\frac{q(N-b)}{(N-q)^+}$, taking $\varepsilon>0$ small enough, such that $\frac{bN}{b+\varepsilon}<N$ and $\frac{Np}{N-b-\varepsilon}<\frac{q(N-b)}{(N-q)^+}$, then we get \eqref{ap1.4} from the above estimates.

Now we prove \eqref{ap1.2}. Note that when $p>0$, for any $\varepsilon>0$, there exists a constant $C_{\varepsilon}>0$, such that for all $a,b\in \R$,
\begin{eqnarray}\label{ap1.6}
\Big||a+b|^{p}-|b|^{p}\Big|\leq \varepsilon |b|^{p}+C_\varepsilon |a|^{p}.
\end{eqnarray}
(See also \cite[Theorem 1.9]{LL}). Taking $a=u_n-u, b=u$ in \eqref{ap1.6}, then
$$\Big||u_n|^{p}-|u|^{p}\Big|\leq \varepsilon |u|^{p}+C_\varepsilon |u_n-u|^{p},$$
thus, using \eqref{ap1.3}, we have 
\begin{align}\label{ap1.7}
\Big|\int |x|^{-b}|u_n|^{p}dx-\int |x|^{-b}|u|^{p}dx \Big|&\leq \int |x|^{-b}\Big||u_n|^{p}-|u|^{p}\Big|dx\nonumber\\
&\leq \varepsilon \int |x|^{-b}|u|^{p}dx+ C_\varepsilon \int |x|^{-b}|u_n-u|^{p}dx \nonumber\\
&=\varepsilon \int |x|^{-b}|u|^{p}dx+o_n(1).
\end{align}
Since $\varepsilon>0$ is arbitrary, then from \eqref{ap1.7}, we obtain \eqref{ap1.2}. Finally, combing \eqref{ap1.2} and \eqref{ap1.4}, then \eqref{ap1.1} follows.

\eqref{ap1.3} has already been proved in the proof of \cite[Lemma 3.1]{BY}, here we will not repeat the same argument.
\end{proof}

\begin{lemma}(\cite[Theorem 3.4]{LL})\label{lm_ap2}
Let $u(x)$ be a nonnegative function on $\R^N$, vanishing at infinity, and let $u^*$ be the Schwartz symmetrization of $u$. Then
\begin{eqnarray}\label{ap1.8}
\int |x|^{-b}|u|^{p}dx\leq \int |x|^{-b}|u^*|^{p}dx, \ b>0,
\end{eqnarray}
with the understanding that when the left side is infinite so is the right side. Moreover, the equality in \eqref{ap1.8} holds if and only if $u=u^*$.
\end{lemma}

In what follows, inspired by \cite[Theorem 8.1.1]{C}, we prove the exponential decay of the solutions of \eqref{eq1.1}, which is interesting by itself, but also may be helpful to verify \eqref{5.14}.
\begin{lemma}\label{th-e-decay}
Assume that $2<p<\frac{q(N-b)}{(N-q)^+}$, with $0 < b < \min\{2,N\}$, $N\geq1$, $q\geq 2 $, and $u$ satisfies Eq. \eqref{eq1.1} with $\lambda <-2$, then the following properties hold:
\begin{itemize}
  \item [(1)] $ u \in W^{2,r} (\mathbb{R}^{N}) $, for every $r\geq 2$. In particular, $ u \in C^{2,\delta} (\mathbb{R}^{N}) $, for all $ \delta \in (0,1) $ and as $ |x| \rightarrow +\infty $,
$$ | D^\beta u(x)| \rightarrow 0 , \quad \forall\ |\beta| \leq 2.$$
  \item [(2)] There holds that
$$ e^{\frac{1}{2} |x|} ( |u(x)| + |\nabla u(x)| ) \in L^\infty (\mathbb{R}^{N}).$$
\end{itemize}
\end{lemma}
\begin{proof}[Proof of Lemma \ref{th-e-decay}]
We mainly use the idea from \cite{C} to prove the theorem. The point $(1)$ follows from the standard elliptic regularity theory, here we omit the details. Now we prove the point $(2)$.

Let $\varepsilon > 0$ and $ \theta_\varepsilon (x) = e ^{\frac{|x|}{1+ \varepsilon |x|}}$, then clearly $\theta_\varepsilon  $ is  bounded, Lipschitz Continuous in $\mathbb{R}^{N}$, and $ |\nabla  \theta_\varepsilon | \leq  \theta_\varepsilon$ a.e. in $\mathbb{R}^{N}$. Multiplying Eq. \eqref{eq1.1} with $ \theta_\varepsilon u$, we have
\begin{equation}\label{dec1}
 \int -\triangle u \cdot ( \theta_\varepsilon u)dx - \int \triangle_q u \cdot ( \theta_\varepsilon u)dx = \int \lambda  \theta_\varepsilon |u|^2dx + \int \theta_\varepsilon \frac{|u|^p}{|x|^b}dx.
 \end{equation}
Note that, by the integrating by parts,
\begin{equation*}
\begin{split}
 \int -\triangle u \cdot ( \theta_\varepsilon u) dx&= \int \nabla u \cdot \nabla( \theta_\varepsilon u)dx\\
                                         &= \int \nabla u \cdot \nabla  \theta_\varepsilon \cdot udx + \int  \theta_\varepsilon |\nabla u |^2dx\\
                                         &\geq -\int |\nabla u||u| \theta_\varepsilon dx + \int  \theta_\varepsilon |\nabla u |^2dx,
\end{split}
\end{equation*}

\begin{equation*}
\begin{split}
 \int -\triangle_q u \cdot ( \theta_\varepsilon u)dx &= \int -div(|\nabla u|^{q-2} \nabla u)\cdot( \theta_\varepsilon u)dx\\
                          &= \int |\nabla u|^{q-2} \nabla u \cdot \nabla ( \theta_\varepsilon u)dx \\
                          &= \int |\nabla u|^{q-2} \nabla u \cdot \nabla  \theta_\varepsilon \cdot udx + \int \theta_\varepsilon |\nabla u|^{q}dx\\
                          &\geq -\int |\nabla u|^{q-1} |u|\theta_\varepsilon dx + \int  \theta_\varepsilon |\nabla u |^qdx.
\end{split}
\end{equation*}
By the Young inequality, we have
\begin{equation*}
 \int |\nabla u| |u|  \theta_\varepsilon dx \leq \frac{1}{2} \int  \theta_\varepsilon |\nabla u|^2dx + \frac{1}{2} \int  \theta_\varepsilon |u|^2dx,
\end{equation*}
\begin{equation*}
  \int |\nabla u|^{q-1} |u|  \theta_\varepsilon dx \leq \frac{q-1}{q} \int  \theta_\varepsilon |\nabla u|^q dx+ \frac{1}{q} \int  \theta_\varepsilon |u|^qdx.
\end{equation*}
Then,
\begin{equation*}
 \int -\triangle u \cdot ( \theta_\varepsilon u)dx \geq \frac{1}{2} \int  \theta_\varepsilon |\nabla u|^2 dx- \frac{1}{2} \int  \theta_\varepsilon |u|^2dx,
\end{equation*}
\begin{equation*}
  \int -\triangle_q u \cdot ( \theta_\varepsilon u)dx \geq \frac{1}{q} \int  \theta_\varepsilon |\nabla u|^q dx- \frac{1}{q} \int  \theta_\varepsilon |u|^qdx.
\end{equation*}
Now, from \eqref{dec1}, we have
\begin{equation*}
  \int  \theta_\varepsilon \frac{|u|^p}{|x|^b}dx +  \lambda \int  \theta_\varepsilon |u|^2dx
  \geq \frac{1}{2} \int  \theta_\varepsilon |\nabla u|^2 dx+ \frac{1}{q} \int  \theta_\varepsilon |\nabla u|^q dx- \frac{1}{2} \int  \theta_\varepsilon |u|^2 dx - \frac{1}{q} \int  \theta_\varepsilon |u|^q dx,
\end{equation*}
which leads to
\begin{equation}\label{dec2}
  \int  \theta_\varepsilon \frac{|u|^p}{|x|^b} dx+ \frac{1}{q} \int  \theta_\varepsilon |u|^q dx
  \geq \frac{1}{2} \int  \theta_\varepsilon |\nabla u|^2 dx+ \frac{1}{q} \int  \theta_\varepsilon |\nabla u|^q dx- (\lambda + \frac{1}{2}) \int  \theta_\varepsilon |u|^2dx.
\end{equation}
By the point $(1)$, we know that $u(x) \rightarrow 0$, as $|x| \rightarrow +\infty $. Thus, there exists $R>1$, such that
$$ |u(x)| \leq 1, \ \mbox{as} \ |x| \geq R. $$
Therefore,
\begin{equation}\label{dec3}
\begin{split}
 \int_{\mathbb{R}^N}  \theta_\varepsilon |u|^q dx&= \int_{|x| \leq R}  \theta_\varepsilon |u|^q dx+ \int_{|x| \geq R}  \theta_\varepsilon |u|^qdx\\
                                        &\leq \int_{|x| \leq R}  \theta_\varepsilon |u|^q dx+\int_{|x| \geq R}  \theta_\varepsilon |u|^2dx\\
                                        &\leq \int_{|x| \leq R}  \theta_\varepsilon |u|^q dx +\int_{\mathbb{R}^N}  \theta_\varepsilon |u|^2dx.
\end{split}
\end{equation}
Similarly,
\begin{equation}\label{dec4}
\begin{split}
 \int_{\mathbb{R}^N}  \theta_\varepsilon \frac{|u|^p}{|x|^b}dx
 &\leq \int_{|x| \leq R}  \theta_\varepsilon \frac{|u|^p}{|x|^b}dx  + \frac{1}{R^b}\int_{\mathbb{R}^N}  \theta_\varepsilon |u|^2 dx\\
 &\leq \int_{|x| \leq R}  \theta_\varepsilon \frac{|u|^p}{|x|^b}dx +\int_{\mathbb{R}^N}  \theta_\varepsilon |u|^2dx.
\end{split}
\end{equation}
Thus taking \eqref{dec3} and \eqref{dec4} into \eqref{dec2},
\begin{equation*}
\begin{split}
 &\frac{1}{2} \int  \theta_\varepsilon |\nabla u|^2dx + \frac{1}{q} \int  \theta_\varepsilon |\nabla u|^qdx - (\lambda + \frac{1}{2}) \int  \theta_\varepsilon |u|^2dx \\
  \leq & \frac{1}{q}\int_{|x| \leq R}  \theta_\varepsilon |u|^q dx+ \int_{|x| \leq R}  \theta_\varepsilon \frac{|u|^p}{|x|^b}dx + (\frac{1}{q} +1) \int_{\mathbb{R}^N}  \theta_\varepsilon |u|^2dx\\
  \leq & \frac{1}{q}\int_{|x| \leq R}  \theta_\varepsilon |u|^q dx+ \int_{|x| \leq R}  \theta_\varepsilon \frac{|u|^p}{|x|^b}dx + \frac{3}{2} \int_{\mathbb{R}^N}  \theta_\varepsilon |u|^2dx.
  \end{split}
\end{equation*}
Then
\begin{equation}\label{dec5}
  \int  \theta_\varepsilon |\nabla u|^2 dx+  \int  \theta_\varepsilon |\nabla u|^qdx - 2(\lambda + 2) \int  \theta_\varepsilon |u|^2dx
  \leq  \int_{|x| \leq R}  \theta_\varepsilon |u|^q dx+ 2 \int_{|x| \leq R}  \theta_\varepsilon \frac{|u|^p}{|x|^b}dx.
\end{equation}
Since $  \theta_\varepsilon = e ^{\frac{|x|}{1+ \varepsilon |x|}} \leq e^{|x|} $, $u \in  C^{2,\delta} (\mathbb{R}^{N})$, $\lambda < -2$ and $0<b<\min\{2,N\}$, then
$$\int_{|x| \leq R}  \theta_\varepsilon |u|^q dx+ 2 \int_{|x| \leq R}  \theta_\varepsilon \frac{|u|^p}{|x|^b}dx < +\infty.$$
Thus, \eqref{dec5} implies that
$$ \max \{\int  \theta_\varepsilon |\nabla u|^2dx, \int  \theta_\varepsilon |\nabla u|^qdx, \int  \theta_\varepsilon |u|^2dx \} \leq \widetilde{C}(u).$$
where $\widetilde{C}(u) > 0$.
Letting $ \varepsilon \rightarrow 0^+$, then we have
\begin{equation}\label{dec6}
\int_{\mathbb{R}^{N}} e^{|x|} |u(x)|^2dx \leq \widetilde{C}(u), \quad \int_{\mathbb{R}^{N}} e^{|x|} |\nabla u(x)|^2dx \leq \widetilde{C}(u).
\end{equation}
and by the point $(1)$, $u$ and $|\nabla u|$ are all globally Lipschitz continuous in $\R^N$, hence we conclude from \eqref{dec6} that
\begin{equation*}
e^{\frac{1}{2}|x|} \Big(|u(x)|+|\nabla u(x)|\Big)\in L^{\infty}(\R^N).
\end{equation*}
which is the point $(2)$.
\end{proof} 

{\bf Acknowledgments}

Tingjian Luo is supported by the National Natural Science Foundation of China (11501137), and the Guangdong Basic and Applied Basic Research Foundation (2025A1515012282).


\end{document}